\documentclass[11pt]{article}
\usepackage[numbers,square]{natbib}
\usepackage{amsmath}
\usepackage{amsthm}
\usepackage{amsfonts}
\usepackage{amssymb}
\usepackage{siunitx}
\usepackage{commath}
\usepackage{graphicx}
\usepackage{xspace}
\usepackage{color}
\usepackage[lofdepth,lotdepth]{subfig}
\usepackage{stmaryrd}
\usepackage{url}
\usepackage{amsthm}
\usepackage{footnote}
\makesavenoteenv{tabular}
\makesavenoteenv{table}
\usepackage{tikz}
\usepackage[hidelinks]{hyperref}
\hypersetup{
  colorlinks   = true, 
  urlcolor     = red, 
  linkcolor    = red, 
  citecolor   = red 
}
\usepackage{cleveref}

\newtheorem{lemma}{Lemma}[section]
\newtheorem{theorem}{Theorem}[section]
\newtheorem{corollary}{Corollary}[section]

\usepackage[margin=0.7in]{geometry}

\newcommand{\tnorm}[1]{{|\!|\!| #1 |\!|\!|}}
\newcommand{\jump}[1]{[\![#1]\!]}
\begin{document}
\begin{table}[tbp]
\begin{center}
\begin{tabular}{c}
  \textsc{A compatible embedded-hybridized discontinuous Galerkin method
  }\\\textsc{for the Stokes--Darcy-transport problem}\\
  \\
  \textsc{Aycil Cesmelioglu\footnotemark\footnotetext{Department of Mathematics and Statistics, Oakland University, Rochester, Michigan, USA. Email: \url{cesmelio@oakland.edu}. ORCID: 0000-0001-8057-6349} and Sander Rhebergen\footnote{Department of Applied Mathematics, University of Waterloo, Waterloo, Ontario, Canada. Email: \url{srheberg@uwaterloo.ca}. ORCID: 0000-0001-6036-0356.}}
\end{tabular}
\end{center}
\end{table}
\subsubsection*{Abstract}
We present a stability and error analysis of an embedded-hybridized
discontinuous Galerkin (EDG-HDG) finite element method for coupled
Stokes--Darcy flow and transport. The flow problem, governed by the
Stokes--Darcy equations, is discretized by a recently introduced
exactly mass conserving EDG-HDG method while an embedded discontinuous
Galerkin (EDG) method is used to discretize the transport equation. We
show that the coupled flow and transport discretization is compatible
and stable. Furthermore, we show existence and uniqueness of the
semi-discrete transport problem and develop optimal a priori error
estimates. We provide numerical examples illustrating the theoretical
results. In particular, we compare the compatible EDG-HDG
discretization to a discretization of the coupled Stokes--Darcy and
transport problem that is not compatible. We demonstrate that where
the incompatible discretization may result in spurious oscillations in
the solution to the transport problem, the compatible discretization
is free of oscillations. An additional numerical example with
realistic parameters is also presented.
\\
\\
\textbf{Keywords:} Stokes--Darcy flow, Coupled flow and transport,
Beavers--Joseph--Saffman, Embedded and Hybridized methods,
Discontinuous Galerkin, Multiphysics.
%
\section{Introduction}
\label{sec:introduction}

The coupled Stokes--Darcy equations describe the interaction between
free flow and flow in porous media. To model the transport of
chemicals and contaminants in, for example, surface/subsurface flows,
biochemical transport, or vascular hemodynamics problems, the
Stokes--Darcy equations are coupled to a transport equation
\cite{Dangelo:2011}. In this paper we consider one way coupling in
which the velocity solution to the Stokes--Darcy flow problem is used
in the transport equation to advect and diffuse a contaminant.

Many different finite element and mixed finite element
\cite{Discacciati:2002, Layton:2002, Burman:2005, Cao:2010,
  Gatica:2009, Camano:2015, Marquez:2015}, discontinuous Galerkin (DG)
\cite{Cesmelioglu:2009, Girault:2009, Girault:2014, Lipnikov:2014,
  Riviere:2005, RiviereYotov:2005}, and hybridizable discontinuous
Galerkin (HDG) \cite{Egger:2013, Fu:2018, Gatica:2017, Igreja:2018}
methods have been proposed to discretize the Stokes--Darcy
problem. Likewise, many different finite element methods have been
proposed to discretize the transport equation, such as streamline
diffusion, DG, and HDG methods \cite{Brooks:1982, Cockburn:1998,
  Houston:2002, Hughes:2005, Nguyen:2009, Wells:2011}. However,
stability and accuracy of a discretization for the Stokes--Darcy
problem and separately a discretization for the transport problem,
does not guarantee that the coupled discretization for the
Stokes--Darcy-transport problem will be stable and accurate. Examples
of discontinuous Galerkin methods for the Stokes--Darcy-transport
problem have been proposed in \cite{Riviere:2014, Vassilev:2009}

To address the issue of coupling the discretization of a flow problem
to the discretization of a transport problem, \citet{Dawson:2004}
introduced the concept of \emph{compatibility}; a discretization for
flow and transport is compatible if it is globally conservative and
zeroth-order accurate. They showed that loss of accuracy and/or loss
of global conservation may occur if the discretization is not
compatible. They furthermore showed that, for discontinuous Galerkin
methods, compatibility is a stronger statement than local conservation
of the flow field.

In this paper we discretize the Stokes--Darcy problem by an
embedded-hybridized discontinuous Galerkin (EDG-HDG) finite element
method \cite{Cesmelioglu:2020}. The two main reasons to consider this
discretization are: (i) mass is conserved exactly, i.e., the velocity
field is divergence-conforming on the whole domain and mass is
conserved point-wise; and (ii) the EDG-HDG discretization has fewer
globally coupled degrees of freedom than traditional DG and HDG
methods and is generally better suited to fast iterative solvers than
HDG methods, as shown in \cite{Rhebergen:2020}. We discretize the
transport equation by an embedded discontinuous Galerkin (EDG) method
\cite{Wells:2011}, the main motivation being that EDG discretizations
are generally computationally more efficient than DG and HDG
discretizations. We will show that the EDG-HDG flow discretization is
compatible with the EDG discretization of the transport equation. We
prove well-posedness of the discrete transport problem and present
optimal error estimates.

The outline for the remainder of this paper is as follows. In
\cref{sec:stokesdarcy} we describe the Stokes--Darcy flow system and
its coupling to the transport problem. The compatible EDG-HDG
discretization for the Stokes--Darcy-transport problem, together with
some of its properties, is introduced in
\cref{sec:numericalmethod}. In \cref{sec:usefulineq} we discuss useful
inequalities that will be used to prove existence and uniqueness of a
solution to the semi-discrete transport problem in
\cref{sec:coerstabexisconsis}. Error estimates are developed for the
semi-discrete transport scheme in \cref{sec:erroranalsemidiscrete}. In
\cref{sec:numexp} we verify the analysis by numerical experiments
while conclusions are drawn in \cref{sec:conclusions}.

\section{The Stokes--Darcy system coupled to transport}
\label{sec:stokesdarcy}
Let $\Omega\subset \mathbb{R}^{\dim}$, ${\dim} = 2, 3$, be a bounded
polygonal domain with boundary $\partial \Omega$ and boundary outward
unit normal vector $n$. We assume the domain is divided into two
non-overlapping polygonal subdomains, $\Omega^s$ and $\Omega^d$, such
that $\Omega=\Omega^s\cup\Omega^d$. The interface separating these
subdomains is denoted by $\Gamma^{I}$. Furthermore, for $j=s,d$, the
outward unit normal vector of $\Omega^{j}$ is denoted by $n^{j}$ and
the exterior boundary of $\Omega^{j}$ is denoted by
$\Gamma^{j}=\partial \Omega\cap \Omega^{j}$. The Stokes--Darcy system
for the velocity $u: \Omega\rightarrow \mathbb{R}^{\dim}$ and pressure
$p : \Omega \rightarrow \mathbb{R}$ is given by
\begin{subequations}
  \begin{align}
    \label{eq:momentum}
    -\nabla\cdot 2\mu\varepsilon(u) + \nabla p &= f^s & & \text{in}\ \Omega^s,
    \\
    \label{eq:d_velocity}
    \kappa^{-1}u + \nabla p &= 0 & & \text{in}\ \Omega^d,
    \\
    \label{eq:mass}
    -\nabla\cdot u &= \chi^d f^d & & \text{in}\ \Omega,
    \\   
    \label{eq:bc_s}
    u &= 0 & & \text{on}\ \Gamma^s,
    \\
    \label{eq:bc_d}
    u\cdot n &= 0 & & \text{on}\ \Gamma^d,
  \end{align}
  \label{eq:system}
\end{subequations}
where $\varepsilon(u) = (\nabla u + \nabla u^T)/2$ is the strain rate
tensor, $\mu>0$ is the constant kinematic viscosity, $\kappa>0$ is the
permeability constant, $f^s: \Omega^s\rightarrow \mathbb{R}^{\dim}$ is
a forcing term, $f^d: \Omega^d\rightarrow \mathbb{R}$ is a source/sink
term, and $\chi^d$ is the characteristic function of $\Omega^d$. In
the following we will denote the restriction of $u$ and $p$ on
$\Omega^{j}$ by $u^{j}$ and $p^{j}$, respectively, for $j=s,d$.

Let $n$ denote the unit normal vector on $\Gamma^I$ pointing outwards
from $\Omega^s$, that is, $n = n^s = -n^d$. We denote the tangential
component of a vector $w$ by $(w)^t := w - (w\cdot n) n$. On the
interface we then prescribe the following transmission conditions:
\begin{subequations}
  \begin{align}
    \label{eq:bc_I_u}
    u^s\cdot n & = u^d\cdot n & & \text{on}\ \Gamma^I,
    \\
    \label{eq:bc_I_p}
    p^s - 2\mu\varepsilon(u^s)n \cdot n & = p^d & & \text{on}\ \Gamma^I,
    \\
    \label{eq:bc_I_slip}
    -2\mu\del{\varepsilon(u^s)n}^t & = \alpha \kappa^{-1/2}(u^s)^t
                              & & \text{on}\ \Gamma^I, 
  \end{align}
  \label{eq:interface}
\end{subequations}
where $\alpha > 0$ is an experimentally determined constant. 
\Cref{eq:bc_I_u,eq:bc_I_p} denote balance of flux and normal stress,
and \cref{eq:bc_I_slip} is the Beavers--Joseph--Saffman interface
condition \cite{Beavers:1967, Saffman:1971}.

The Stokes--Darcy system described above is coupled to a transport
equation over the time interval of interest $I=(0,T]$. Given a
porosity constant $\phi$ such that $0< \phi \leq 1$ in $\Omega^d$ and
$\phi=1$ in $\Omega^s$, velocity field
$u : \Omega\to\mathbb{R}^{\dim}$, and source/sink term
$f : \Omega\times I \rightarrow \mathbb{R}$, the transport equation
for the concentration $c : \Omega \times I \rightarrow \mathbb{R}$ of
a contaminant is given by:
\begin{subequations}
  \begin{align}
    \label{eq:transport_eq}
    \phi\partial_t c + \nabla \cdot (c u -D(u) \nabla c ) & = f && 
    \text{in}\ \Omega \times I,
    \\
    \label{eq:transport_bc}
    D(u) \nabla c \cdot n & = 0 &&\text{on}\ \partial \Omega \times I,
    \\
    \label{eq:transport_ic}
    c(x,0) & = c_0(x) && \text{in}\ \Omega, 
  \end{align}
  \label{eq:transport_system}
\end{subequations}
where $c_0 : \Omega \rightarrow \mathbb{R}$ is a suitably smooth
initial condition, and $D(u)$ is the diffusion/dispersion tensor. Let
$|\cdot|$ denote the Euclidean norm. We will assume that $D(u)$
satisfies the following conditions for $u, v \in \mathbb{R}^{\dim}$:
\begin{subequations}
  \begin{align}
  \label{eq:D_min}
    D_{\min} |x|^2 &\le D(u) x \cdot x \quad \forall x \in \mathbb{R}^{\dim},
    \\
    \label{eq:upperboundDuabs}
    |D(u)| &\le C(1 + \envert{u}),
    \\
    \label{eq:DLipschitz}
    \envert{ D(u) - D(v) } &\le C \envert{u - v},    
  \end{align}
\end{subequations}
where $D_{\min} > 0$ and $C > 0$ a constant. For the remainder of this
paper, $C > 0$ will always denote a generic constant.

\section{The numerical method}
\label{sec:numericalmethod}
Before introducing the numerical method we consider two properties of
the coupled Stokes--Darcy flow and transport system. First,
integrating the transport equation \cref{eq:transport_eq} over
$\Omega \times (0, t)$, applying the boundary conditions
\cref{eq:bc_s,eq:bc_d,eq:transport_bc}, and the initial condition
\cref{eq:transport_ic}, results in the following expression of global
conservation:
\begin{equation*}
  \label{eq:globalconservation}
  \int_{\Omega} \phi c(x, t) \dif x 
  = \int_{\Omega} \phi c_0(x) \dif x + \int_0^t\int_{\Omega} f \dif x \dif s.
\end{equation*}
Second, if the initial condition in \cref{eq:transport_ic} is set as
$c_0(x) = \tilde{c}$, with $\tilde{c}$ a constant, and the source/sink
term $f$ in \cref{eq:transport_eq} takes the form
$f = -\chi^df^d\tilde{c}$, then $c(x,t) = \tilde{c}$ for $t > 0$.

A discretization of Stokes--Darcy flow \cref{eq:system,eq:interface}
coupled to transport \cref{eq:transport_system} is called
\emph{compatible} if it satisfies these two properties at the discrete
level \cite{Dawson:2004}. Compatibility was shown in
\cite{Dawson:2004} to be a desirable property of a discretization that
couples a flow model to transport for reasons of accuracy and/or
stability. They furthermore showed that for discontinuous Galerkin
discretizations, compatibility is a \emph{stronger statement} than
local conservation of the flow field.

In this section we present a compatible embedded-hybridized
discontinuous Galerkin method for the Stokes--Darcy flow coupled to
transport \cref{eq:system,eq:interface,eq:transport_system}.

\subsection{Notation}
For $j=s,d$, let $\mathcal{T}^j := \cbr{K}$ be a shape-regular
triangulation of $\Omega^j$ consisting of non-overlapping elements $K$
and such that meshes match at the interface $\Gamma^I$. Furthermore,
let $\mathcal{T} = \mathcal{T}^s\cup\mathcal{T}^d$. We denote the
diameter of an element $K$ by $h_K$ and set
$h := \max_{K\in\mathcal{T}}h_K$. The boundary of an element is
denoted by $\partial K$ and the element boundary outward unit normal
vector is denoted by $n$.

An interior facet is a facet that is shared by two adjacent
elements. A boundary facet is a facet of $\partial K$ that lies on
$\partial\Omega$. The set and union of all facets are denoted by,
respectively, $\mathcal{F} = \cbr{F}$ and $\Gamma_0$. Furthermore,
the set of all facets that lie on the interface $\Gamma^I$ is denoted
by $\mathcal{F}^I$, and by $\mathcal{F}^j$ and $\Gamma_0^j$ ($j=s,d$)
we denote, respectively, the set and union of all facets in
$\overline{\Omega}^j$.

We use $(v,w)_D$ to denote the $L^2$-inner product of two functions
$v$ and $w$ defined on $D \subset \mathbb{R}^{\dim}$. By
$\langle v, w\rangle_D$ we denote the $L^2$-inner product of two
functions $v$ and $w$ defined on $D \subset \mathbb{R}^{\rm
  dim-1}$. For $D \subset \mathbb{R}^{\dim}$ or
$D \subset \mathbb{R}^{\dim - 1}$ we denote the standard norm on
the Sobolev space $W^{s,p}(D)$ by $\norm{\cdot}_{s,p,D}$ and the
semi-norm on $W^{s,p}(D)$ by $|\cdot|_{s,p,D}$. When $p=2$, we drop the
subscript $p$ and write for the norm and semi-norm $\norm{\cdot}_{s,D}$
and $|\cdot|_{s,D}$, respectively. For the $L^2$-norm we drop the
subscript $s=0$ and write $\norm{\cdot}_{D}$.

\subsection{The embedded-hybridized DG method for the Stokes--Darcy
  equations}
\label{ss:hdgflow}
We discretize the Stokes--Darcy flow problem
\cref{eq:system,eq:interface} by an exactly mass conserving
embedded-hybridized discontinuous Galerkin (EDG-HDG) method
\cite{Cesmelioglu:2020}. For this we consider the following
discontinuous Galerkin finite element function spaces on $\Omega$,
\begin{equation*}
  \label{eq:DGcellspaces}
  \begin{split}
    V_h &:= \cbr[1]{v_h \in \sbr[0]{L^2(\Omega)}^{\dim} : \ v_h \in
      \sbr[0]{P_k(K)}^{\dim} \ \forall\ K \in \mathcal{T}},
    \\
    Q_h &:= \cbr[1]{q_h \in L^2(\Omega) : \ q_h \in P_{k-1}(K) \
      \forall \ K \in \mathcal{T}}\cap L^2_0(\Omega),
    \\
    Q_h^j &:= \cbr[1]{q_h \in L^2(\Omega^j) : \ q_h \in P_{k-1}(K) \
      \forall \ K \in \mathcal{T}^j},\quad j=s,d,
  \end{split}
\end{equation*}
where $P_k(D)$ denotes the space of polynomials of degree $k$ on
domain~$D$ and $L^2_0(\Omega) := \{q \in L^2(\Omega) : \int_{\Omega} q
\dif x = 0\}$. On $\Gamma_0^s$ and $\Gamma_0^d$, we consider the finite
element spaces:
\begin{equation*}
  \label{eq:DGfacetspaces}
  \begin{split}
    \bar{V}_h &:= \cbr[1]{\bar{v}_h \in \sbr[0]{L^2(\Gamma_0^s)}^{\dim}:\
      \bar{v}_h \in \sbr[0]{P_{k}(F)}^{\dim}\ \forall\ F \in \mathcal{F}^s,\
      \bar{v}_h = 0 \ \mbox{on}\ \Gamma^s} \cap \sbr[0]{C^0(\Gamma_0^s)}^{\dim},
    \\
    \bar{Q}_h^j &:= \cbr[1]{\bar{q}_h^j \in L^2(\Gamma_0^j) : \ \bar{q}_h^j
      \in P_{k}(F) \ \forall\ F \in \mathcal{F}^j},\quad j=s,d.
  \end{split}
\end{equation*}
Note that functions in $\bar{V}_h$ are continuous on $\Gamma_0^s$,
while functions in $\bar{Q}_h^j$ are discontinuous on $\Gamma_0^j$,
for $j=s,d$.

Following the notation of \cite{Cesmelioglu:2020}, we introduce the
spaces $\boldsymbol{V}_h := V_h \times \bar{V}_h$,
$\boldsymbol{Q}_h := Q_h \times \bar{Q}_h^s \times \bar{Q}_h^d$ and
$\boldsymbol{Q}_h^{j} := Q_h^j \times \bar{Q}_h^j$ for $j=s,d$. We
denote function pairs in $\boldsymbol{V}_h$, $\boldsymbol{Q}_h$ and
$\boldsymbol{Q}_h^{j}$, for $j=s,d$, by
$\boldsymbol{v}_h := (v_h, \bar{v}_h) \in \boldsymbol{V}_h$,
$\boldsymbol{q}_h := (q_h, \bar{q}_h^s, \bar{q}_h^d) \in
\boldsymbol{Q}_h$ and
$\boldsymbol{q}_h^{j} := (q_h, \bar{q}_h^j) \in
\boldsymbol{Q}_h^{j}$. Finally, we set
$\boldsymbol{X}_h := \boldsymbol{V}_h \times \boldsymbol{Q}_h$.

The EDG-HDG method for the Stokes--Darcy flow problem now reads: find
$(\boldsymbol{u}_h, \boldsymbol{p}_h) \in \boldsymbol{X}_h$ such that
\begin{equation}
  \label{eq:hdgwf}
  F_h( (\boldsymbol{u}_h, \boldsymbol{p}_h), (\boldsymbol{v}_h, \boldsymbol{q}_h) )  
  = (f^s, v_h)_{\Omega^s} + (f^d, q_h)_{\Omega^d} 
  \quad \forall (\boldsymbol{v}_h, \boldsymbol{q}_h) \in \boldsymbol{X}_h,
\end{equation}
where
\begin{equation*}
  F_h( (\boldsymbol{u}, \boldsymbol{p}), (\boldsymbol{v}, \boldsymbol{q}) ) =
  a_h(\boldsymbol{u}, \boldsymbol{v}) 
  + b_h (\boldsymbol{p}, \boldsymbol{v}) + b_h(\boldsymbol{q}, \boldsymbol{u}).
\end{equation*}
Here the bi-linear form $a_h(\cdot, \cdot)$ is defined as
\begin{equation*}
  \begin{split}
    a_h(\boldsymbol{u}, \boldsymbol{v}) =&
    \sum_{K\in \mathcal{T}^s}(2\mu \varepsilon(u),\varepsilon(v))_{K}
    +\sum_{K\in \mathcal{T}^s}\langle \tfrac{2\beta_f\mu}{h_K}(u-\bar{u}), v-\bar{v}\rangle_{\partial K}
    \\
    & -\sum_{K\in \mathcal{T}^s}\langle 2\mu\varepsilon(u)n^s,v-\bar{v}\rangle_{\partial K}
    - \sum_{K\in \mathcal{T}^s}\langle2\mu\varepsilon(v)n^s, u-\bar{u}\rangle_{\partial K}
    \\
    & + (\kappa^{-1} u, v)_{\Omega^d} 
    + \langle\alpha\kappa^{-1/2} \bar{u}^t, \bar{v}^t\rangle_{\Gamma^I},    
  \end{split}
\end{equation*}
where $\beta_f > 0$ is a penalty parameter. The bi-linear form
$b_h(\cdot, \cdot)$ is defined as
\begin{equation*}
  b_h(\boldsymbol{p}, \boldsymbol{v} )
  =
  -\sum_{K\in \mathcal{T}}(p, \nabla \cdot v)_{K} 
  + \sum_{j=s,d} \sum_{K\in \mathcal{T}^j}\langle \bar{p}^j, v \cdot n^j\rangle_{\partial K}
  - \langle\bar{p}^s-\bar{p}^d,\bar{v}\cdot n\rangle_{\Gamma^I}.
\end{equation*} 

The following results are from \cite{Cesmelioglu:2020} and will be
used in the analysis. For sufficiently large $\beta_f$, there exists a
unique solution
$(\boldsymbol{u}_h, \boldsymbol{p}_h) \in \boldsymbol{X}_h$ to
\cref{eq:hdgwf} (see \cite[Proposition 1]{Cesmelioglu:2020}). An a
priori error analysis showed that if the velocity solution $u$ to
\cref{eq:system,eq:interface} satisfies $u^s \in H^{k+1}(\Omega^s)$
and $u^d\in H^{k+1}(\Omega^d)$ with $k\geq 1$, then \cite[Theorem
3]{Cesmelioglu:2020}
\begin{equation}
  \label{eq:L2error-u}
  \norm{u-u_h}_{\Omega} \le Ch^{k+1},
\end{equation}
where $C$ is a generic constant that depends on the regularity of
$u^s$ and $u^d$. Furthermore, the EDG-HDG method for the Stokes--Darcy
system is exactly mass conserving, divergence-conforming, and
satisfies \cref{eq:bc_I_u}, i.e.,
\begin{subequations}
  \begin{align}
    \label{eq:uh_properties_a}
    -\nabla \cdot u_h & = \chi^d\Pi_Q f^d && \forall x \in K,\ \forall K \in \mathcal{T},
    \\
      \label{eq:uh_properties_b}
    \jump{u_h\cdot n} & = 0 && \forall x \in F,\ \forall F \in \mathcal{F},
    \\
    \label{eq:uh_properties_c}
    u_h \cdot n & = \bar{u}_h \cdot n && \forall x \in F,\ \forall F \in \mathcal{F}^I,    
  \end{align}
  \label{eq:uh_properties}
\end{subequations}
where $\Pi_Q$ is the standard $L^2$-projection into $Q_h$ and
$\jump{\cdot}$ is the usual jump operator.

\subsection{The embedded DG method for the transport equation}
\label{sec:stokesdarcytransport}
Before introducing the embedded discontinuous Galerkin (EDG) method
for the transport equation, we first replace the exact velocity $u$ in
\cref{eq:transport_system} by the discrete velocity $u_h$:
\begin{subequations}
  \begin{align}
    \label{eq:transport_eq_uh}
    \phi\partial_t c + \nabla \cdot (c u_h -D(u_h) \nabla c ) & = f && 
    \text{in}\ \Omega \times I,
    \\
    \label{eq:transport_bc_uh}
    D(u_h) \nabla c \cdot n & = 0 &&\text{on}\ \partial \Omega\times I,
    \\
    \label{eq:transport_ic_uh}
    c(x,0) & = c_0(x) && \text{in}\ \Omega.
  \end{align}
  \label{eq:transport_system_uh}
\end{subequations}
We will introduce the EDG method for
\cref{eq:transport_system_uh}. For this we require the following
discrete spaces:
\begin{equation}
  \label{eq:spaces_transport}
  \begin{split}
    C_h &= \cbr[0]{c_h\in L^2(\Omega) : \ c_h \in P_{\ell}(K) ,\
      \forall \ K \in \mathcal{T}},
    \\
    \bar{C}_h &= \cbr[0]{\bar{c}_h \in L^2(\Gamma_0) : \ \bar{c}_h \in
      P_{\ell}(F) \ \forall\ F \in \mathcal{F}} \cap C^0(\Gamma_0),
  \end{split}
\end{equation}
where the choice of $\ell$ will be discussed in
\cref{ss:compatibility}. For notational purposes, we introduce
$\boldsymbol{C}_h = C_h \times \bar{C}_h$ and
$\boldsymbol{c}_h = (c_h,\bar{c}_h) \in \boldsymbol{C}_h$.

The semi-discrete EDG method for the transport problem
\cref{eq:transport_system_uh} is now given by: For each $t > 0$, find
$\boldsymbol{c}_h(t)\in \boldsymbol{C}_h$ such that
\begin{equation}
  \label{eq:semidiscrete}
  \sum_{K\in \mathcal{T}}(\phi\, \partial_tc_h(t), w_h)_{K} 
  + B_h(u_h; \boldsymbol{c}_h(t), \boldsymbol{w}_h)
  = \sum_{K\in \mathcal{T}} (f(t) , w_h)_K \quad
  \forall \boldsymbol{w}_h\in \boldsymbol{C}_h,
\end{equation}
where
\begin{equation}
  \label{eq:bilinearform_Bh}
  B_h(u; \boldsymbol{c}(t), \boldsymbol{w}) =
  B_h^a(u; \boldsymbol{c}(t), \boldsymbol{w})
  + B_h^d(u; \boldsymbol{c}(t), \boldsymbol{w}).
\end{equation}
Here $B_h^a(u; \boldsymbol{c}(t), \boldsymbol{w})$ and
$B_h^d(u; \boldsymbol{c}(t), \boldsymbol{w})$ represent, respectively
the advective and diffusive parts of the bi-linear form. They are
defined as:
\begin{equation}
  \label{eq:bilinearform_Bha}
  B_h^a(u; \boldsymbol{c}(t), \boldsymbol{w})
  = - \sum_{K\in \mathcal{T}} (c(t) \, u , \nabla w)_{K}
  + \sum_{K\in \mathcal{T}}\langle c(t)\, u \cdot n , w-\bar{w}\rangle_{\partial K}
  - \sum_{K\in \mathcal{T}}\langle u\cdot n\, (c(t)-\bar{c}(t)), w-\bar{w}\rangle_{\partial K^{\rm in}},  
\end{equation}
where $\partial K^{\rm in}$ denotes the portion of the boundary where
$u_h \cdot n < 0$, and
\begin{equation*}
  \label{eq:bilinearform_Bhd}
  \begin{split}
    B_h^d(u; \boldsymbol{c}(t), \boldsymbol{w}) = &
    \sum_{K\in \mathcal{T}}(D(u)\nabla c(t) , \nabla w)_{K}
    - \sum_{K\in \mathcal{T}}\langle[D(u)\nabla c(t)] \cdot n, w-\bar{w}\rangle_{\partial K}
    \\
    & + \sum_{K\in \mathcal{T}} \tfrac{\beta_c }{h_K} \langle[D(u) n](c(t)-\bar{c}(t)),(w-\bar{w}) n\rangle_{\partial K}
    \\
    & - \sum_{K\in \mathcal{T}} \langle [D(u)\nabla w] \cdot n , c(t)-\bar{c}(t)\rangle_{\partial K},    
  \end{split}
\end{equation*}
where $\beta_c > 0$ is a penalty parameter.

To complete the discretization, we impose the initial condition
\cref{eq:transport_ic} by an $L^2$-projection of $c_0$ into $C_h$.

\subsection{Compatibility}
\label{ss:compatibility}
In this section we show that \cref{eq:hdgwf} and
\cref{eq:semidiscrete} describe a compatible discretization of the coupled
Stokes--Darcy flow and transport problem provided $\ell$ in
\cref{eq:spaces_transport} is suitably chosen. Since global
conservation of the EDG method for the transport equation was shown in
\cite{Wells:2011}, we only show that \cref{eq:semidiscrete} is able to
preserve the constant solution when $f = -\chi^df^d\tilde{c}$.

The constant $\boldsymbol{c}_h = (\tilde{c}, \tilde{c})$ is preserved
by \cref{eq:semidiscrete} if and only if
\begin{equation}
  \label{eq:semidiscreteconserveconstant}
  B_h(u_h; (\tilde{c}, \tilde{c}), \boldsymbol{w}_h)
  =
  -\sum_{K\in \mathcal{T}^d} (f^d\tilde{c} , w_h)_K
  \quad
  \forall \boldsymbol{w}_h\in \boldsymbol{C}_h.
\end{equation}
We observe that \cref{eq:semidiscreteconserveconstant} is equivalent
to
\begin{equation}
  \label{eq:claimsr}
  -\sum_{K\in\mathcal{T}} (u_h , \nabla w_h)_K
  +\sum_{K\in\mathcal{T}} \langle u_h\cdot n, w_h-\bar{w}_h\rangle_{\partial K}
  =
  -\sum_{K\in\mathcal{T}^d} (f^d, w_h)_K
  \quad
  \forall \boldsymbol{w}_h\in \boldsymbol{C}_h.  
\end{equation}
Using that
$\nabla\cdot (u_hw_h) = u_h\cdot \nabla w_h + w_h\nabla\cdot u_h$ on
each element $K$, integration by parts, single-valuedness of
$\bar{w}_h$ and $u_h\cdot n$ on interior facets (by
\cref{eq:uh_properties_b}), and that $u_h\cdot n = 0$ on the boundary
of the domain, \cref{eq:claimsr} simplifies to
\begin{equation*}
  \sum_{K\in\mathcal{T}} (\nabla \cdot u_h, w_h)_K
  =
  -\sum_{K\in\mathcal{T}^d} (f^d, w_h)_K
  \quad
  \forall w_h \in C_h,
\end{equation*}
and so, using \cref{eq:uh_properties_a}, the constant
$\boldsymbol{c}_h = (\tilde{c}, \tilde{c})$ is preserved by
\cref{eq:semidiscrete} if and only if
\begin{equation*}
  \label{eq:fdQfd}
  \sum_{K\in\mathcal{T}^d} (\Pi_Qf^d, w_h)_K
  =
  \sum_{K\in\mathcal{T}^d} (f^d, w_h)_K
  \quad
  \forall w_h \in C_h.      
\end{equation*}
This statement implies that if $f^d \ne 0$ we must choose $\ell = k-1$
in \cref{eq:spaces_transport} for the discretization defined by \cref{eq:hdgwf} and
\cref{eq:semidiscrete} to be a compatible discretization of the
coupled Stokes--Darcy flow and transport problem. If $f^d = 0$ then
$\ell$ can be chosen independent of $k$.

\section{Useful inequalities}
\label{sec:usefulineq}
In subsequent sections, extensive use will be made of the following
continuous trace inequalities \cite[Theorem 1.6.6]{Brenner:book}:
\begin{align}
  \label{eq:2-trace-continuous}
  \norm{v}_{\partial K}^2 & \le C \del[1]{h_K^{-1} \norm{v}_K^2 + h_K \envert{v}_{1,K}^2} && \forall v\in H^1(K),
  \\
  \label{eq:inf-trace-continuous}
  \norm{v}_{0,\infty,\partial K} & \leq C\norm{v}_{0,\infty,K} &&\forall v\in W^{1,\infty}(K),
\end{align}
as well as the following discrete inverse and trace
inequalities \cite[Lemma 1.50, Lemma 1.52]{Pietro:book}
\begin{align}
  \label{eq:inverse}
  \norm{v_h}_{0,\infty, K} & \leq C h_K^{-{\dim}/2}\norm{v_h}_K && \forall v_h \in P_k(K),
  \\
  \label{eq:2-trace}
  \norm{v_h}_{\partial K} & \leq C h_K^{-1/2}\norm{v_h}_{K} && \forall v_h \in P_k(K).
\end{align}
The inequalities \eqref{eq:2-trace-continuous}--\eqref{eq:2-trace} hold
also for $\rm dim$-dimensional vector functions.

An immediate consequence of
\cref{eq:2-trace-continuous,eq:2-trace,eq:L2error-u} is the following
trace inequality that holds for the velocity solution $u$ to
\cref{eq:system,eq:interface} and velocity solution $u_h$ to
\cref{eq:hdgwf}: For $u^s \in \sbr[0]{H^{k+1}(\Omega^s)}^{\dim}$ and
$u^d \in \sbr[0]{H^{k+1}(\Omega^d)}^{\dim}$ with $k\geq 1$,
\begin{equation}
  \label{eq:bounduhuboundary}
  \norm{u-u_h}_{\partial K} \le Ch_K^{k+1/2}.
\end{equation}  
By $\Pi_V$ we denote the $L^2$-projection onto $V_h$ and recall that
for all $0\leq s\leq k+1$ and $u \in \sbr[0]{W^{s,p}(K)}^{\dim}$,
$k\ge 0$ \cite[Proposition 1.135]{Ciarlet:book}:
\begin{equation} 
  \label{eq:L2u} 
  \norm{u-\Pi_V u}_{0,p,K} \leq C h_K^{s}\norm{u}_{s,p,K},
 \quad 1\leq p \leq \infty.
\end{equation}
We also require a bound on $\norm{u_h}_{0,\infty,\Omega}$ which we
prove next.

\begin{lemma}
  Let $u$ denote the velocity solution to
  \cref{eq:system,eq:interface} and assume
  $u\in \sbr[0]{L^{\infty}(\Omega)}^{\dim}$ such that
  $u^s\in \sbr[0]{H^{k+1}(\Omega^s)}^{\dim}$ and
  $u^d\in \sbr[0]{H^{k+1}(\Omega^d)}^{\dim}$ with $k\geq 1$. Then the
  velocity solution $u_h$ to \cref{eq:hdgwf} satisfies
  \begin{equation}
    \label{eq:uhbound-Linf}
    \norm{u_h}_{0,\infty,\Omega}\leq C.
  \end{equation}
\end{lemma}
\begin{proof} 
  By \cref{eq:inverse}, \cref{eq:L2error-u}, and \cref{eq:L2u} we
  find that
  \begin{align*}
    \norm{u_h}_{0,\infty,K}
    &\leq \norm{u_h-\Pi_V u}_{0,\infty,K} + \norm{\Pi_V u}_{0,\infty,K}
    \\
    & \leq C h_K^{-{\dim}/2}\norm{u_h-\Pi_V u}_{K} + \norm{\Pi_V u}_{0,\infty,K}
    \\
    & \leq Ch_K^{-{\dim}/2}\del[1]{\norm{u_h-u}_{K} + \norm{u - \Pi_V u}_{K}} + \norm{u}_{0,\infty,K}
    \\
    & \leq Ch^{k+1-{\dim}/2}(\norm{u}_{k+1,\Omega^s} + \norm{u}_{k+1,\Omega^d}) + C\norm{u}_{0,\infty,K}.
  \end{align*}
  The result follows by taking the maximum over $K\in \mathcal{T}$.
\end{proof}

Finally, by \cref{eq:upperboundDuabs}, we note that for
$v \in \sbr[0]{L^{\infty}(\Omega)}^{\dim}$,
\begin{equation}
  \label{eq:D_max}
  \norm[0]{D(v)}_{0,\infty, \Omega} \leq D_{\max},
\end{equation}
where the constant $D_{\max}>0$ depends on
$\norm{v}_{0,\infty,\Omega}$, and that for
$v \in \sbr[0]{W^{1,\infty}(K)}^{\dim}$, $K\in\mathcal{T}$,
\begin{equation}
  \label{eq:Dboundary}
  \norm[0]{D(v)}_{0,\infty,\partial K} \leq C(1+\norm[0]{v}_{0,\infty,\partial K}) \le C(1+\norm{v}_{0,\infty,K}).
\end{equation}

\section{Stability, existence and consistency of the EDG method for
  the transport equation}
\label{sec:coerstabexisconsis}
In this section we show stability and existence of the solution to the
semi-discrete EDG method of the transport equation
\cref{eq:semidiscrete} as well as consistency of the method. For this we define the
following semi-norm on $\boldsymbol{C}_h$:
\begin{equation*}
  \tnorm{\boldsymbol{w}_h}_{c}^2 =
  \sum_{K\in \mathcal{T}}\del[1]{\, \norm{\nabla w_h}_K^2
    + h_K^{-1}\norm{w_h-\bar{w}_h}_{\partial K}^2}.
\end{equation*}
We first prove useful properties of the advective and diffusive parts
of the bi-linear form \cref{eq:bilinearform_Bh}.

\begin{lemma}
  \label{lem:coercivityBha}
  Let $u_h \in V_h$ be the velocity solution to \cref{eq:hdgwf}. Then
  for all $\boldsymbol{w}_h \in \boldsymbol{C}_h$,
  \begin{equation*}
    \label{eq:coercivityBha}
    B_h^a(u_h; \boldsymbol{w}_h,\boldsymbol{w}_h) = \frac{1}{2} \sum_{K\in \mathcal{T}^d} (\nabla \cdot u_h, w_h^2)_K +
    \frac{1}{2} \sum_{K\in \mathcal{T}} \norm[0]{ |u_h\cdot n|^{1/2}(w_h-\bar{w}_h) }^2_{\partial K}.
  \end{equation*}
\end{lemma}
\begin{proof}
  By definition of $B_h^a$ \cref{eq:bilinearform_Bha},
  \begin{multline*}
      B_h^a(u_h; \boldsymbol{w}_h,\boldsymbol{w}_h) =
      -\sum_{K\in \mathcal{T}} (w_h, u_h \cdot \nabla w_h)_K
      + \sum_{K\in \mathcal{T}} \langle (u_h \cdot n) \,w_h ,w_h-\bar{w}_h\rangle_{\partial K}
      \\
      - \sum_{K\in \mathcal{T}} \langle u_h\cdot n(w_h-\bar{w}_h), w_h-\bar{w}_h\rangle_{\partial K^{\rm in}}.      
  \end{multline*}
  Using integration by parts for the first term on the right hand
  side, the algebraic identity $a(a-b)=\frac12 (a^2-b^2+(a-b)^2)$, and
  \cref{eq:uh_properties}, we find
  \begin{equation*}
    \begin{split}
      B_h^a(u_h; \boldsymbol{w}_h, &\boldsymbol{w}_h) = 
      - \frac{1}{2} \sum_{K\in \mathcal{T}} \langle w_h^2, u_h\cdot n\rangle_{\partial K} 
      +\frac{1}{2} \sum_{K\in \mathcal{T}^d} (\nabla \cdot u_h, w_h^2)_K
      \\
      +& \frac{1}{2} \sum_{K\in \mathcal{T}} \langle u_h \cdot n,w_h^2-\bar{w}_h^2 
      + (w_h-\bar{w}_h)^2\rangle_{\partial K}   
      - \sum_{K\in \mathcal{T}} \langle u_h\cdot n,(w_h-\bar{w}_h)^2\rangle _{\partial K^{\rm in}}.
    \end{split}
  \end{equation*}
  The result follows using the single valuedness of $\bar{w}_h$ and
  $u_h\cdot n$ on element boundaries, and that
  $|u_h\cdot n|=u_h\cdot n$ on
  $\partial K\backslash\partial K^{\rm in}$ and
  $|u_h\cdot n|=-u_h\cdot n$ on $\partial K^{\rm in}$.
\end{proof}
\begin{lemma}[coercivity of $B_h^d$]
  \label{lem:coercivity_Bhd}
  Let $u$ denote the velocity solution to
  \cref{eq:system,eq:interface} and assume that
     $u\in \sbr[0]{W^{1,\infty}(\Omega)}^{\dim}$ such that
  $u^s\in \sbr[0]{H^{k+1}(\Omega^s)}^{\dim}$ and
  $u^d \in \sbr[0]{H^{k+1}(\Omega^d)}^{\dim}$ with $k\geq 1$. Let
  $u_h\in V_h$ be the velocity solution to \cref{eq:hdgwf}. There
  exists a constant $\beta_{c,0} > 0$ such that if
  $\beta_c>\beta_{c,0}$, then for all
  $\boldsymbol{w}_h \in \boldsymbol{C}_h$
  \begin{equation*}
    \label{eq:coercivityBhd}
    B_h^d(u_h; \boldsymbol{w}_h,\boldsymbol{w}_h) \geq C\tnorm{\boldsymbol{w}_h}_c^2.
  \end{equation*}
\end{lemma}
\begin{proof}
  The proof is similar to \cite[Lemma 5.2]{Wells:2011} but in addition
  using \cref{eq:D_min}, \cref{eq:uhbound-Linf}, \cref{eq:D_max} and
  \cref{eq:Dboundary} to bound the terms involving $D(u_h)$.  
\end{proof}
By \cref{lem:coercivityBha,lem:coercivity_Bhd} we may now conclude
that when the velocity solution to \cref{eq:system,eq:interface}
satisfies $u\in \sbr[0]{W^{1,\infty}(\Omega)}^{\dim}$ such that
$u^s\in \sbr[0]{H^{2}(\Omega^s)}^{\dim}$ and
$u^d\in \sbr[0]{H^{2}(\Omega^d)}^{\dim}$, and when $u_h \in V_h$ is
the solution to \cref{eq:hdgwf}, then
\begin{equation}
  \label{eq:coercivity}
  B_h(u_h; \boldsymbol{w}_h,\boldsymbol{w}_h)\geq C\tnorm{\boldsymbol{w}_h}_c^2 
  + \frac{1}{2} \sum_{K\in \mathcal{T}^d} (\nabla \cdot u_h, w_h^2)_K.
\end{equation}

In what follows we will require the following modification of the
classical Gr\"{o}nwall's lemma \cite[Corollary 1.2]{Bainov:book}.
\begin{lemma}
  \label{lem:classicGronwall}
  Let $f(t), g(t), h(t)$ be continuous functions defined on $[0,T]$
  such that $g(t)$ is non-negative and $h(t)$ is non-decreasing in
  $[0,T]$. Let $k$ be a non-negative constant. If
  \begin{equation*}
    f(t) + \int_0^t g(s) \dif s \leq h(t) + \int_0^t k f(s)\dif s \quad \forall t\in [0,T],
  \end{equation*}
  then
  \begin{equation*}
    f(t) + \int_0^t g(s) \dif s \leq h(t)e^{kt} \quad \forall t\in [0,T].
  \end{equation*}
\end{lemma}
\begin{lemma}[stability]
  \label{lem:stability} 
  Assume that $u$, the velocity solution to
  \cref{eq:system,eq:interface}, satisfies
  $u\in \sbr[0]{W^{1,\infty}(\Omega)}^{\dim}$ such that
  $u^s\in \sbr[0]{H^{k+1}(\Omega^s)}^{\dim}$,
  $u^d\in \sbr[0]{H^{k+1}(\Omega^d)}^{\dim}$, $k\geq 1$ and
  $\nabla \cdot u^d\in L^{\infty}(\Omega^d)$. Let
  $c_0\in L^2(\Omega).$ Then the solution
  $\boldsymbol{c}_h\in \boldsymbol{C}_h$ to \cref{eq:semidiscrete}
  satisfies
  \begin{equation*}
    \norm[0]{c_h(t)}^2_{\Omega}
    + \int_0^{t}\tnorm{\boldsymbol{c}_h(s)}_c^2 \dif s
    \leq C\del[2]{ \norm{c_0}^2_{\Omega}
      + \int_0^{t} \norm{f(s)}_{\Omega}^2 \dif s}
    \quad \forall t \in [0, T].
  \end{equation*}
\end{lemma}
\begin{proof}
  Take $\boldsymbol{w}_h=\boldsymbol{c}_h(t)$ in \cref{eq:semidiscrete}.
  Then by \cref{eq:coercivity} and the Cauchy--Schwarz inequality,
  \begin{equation}
    \label{eq:stability-aux}
    \begin{split}
      \tfrac{\phi}{2}\od{}{t}\norm[0]{c_h(t)}^2_{\Omega}
      + C\tnorm{\boldsymbol{c}_h(t)}_c^2
      & \leq \sum_{K\in \mathcal{T}}(f(t),c_h(t))_K - \frac{1}{2}\sum_{K\in \mathcal{T}^d}(\nabla \cdot u_h,c_h^2(t))_K
      \\    
      & \leq \norm[0]{f(t)}_{\Omega} \norm[0]{c_h(t)}_{\Omega}
      + \tfrac{1}{2}\norm[0]{\nabla \cdot u_h}_{0,\infty,\Omega^d}\norm[0]{c_h}_{\Omega}^2.
    \end{split}
  \end{equation}
  Note that by \cref{eq:uh_properties},
  $\norm{\nabla \cdot u_h}_{0,\infty,\Omega^d} = \norm[0]{\Pi_Q
    f^d}_{0,\infty,\Omega^d} \leq C\norm[0]{f^d}_{0,\infty,\Omega^d}$.
  Combining with \cref{eq:stability-aux}, integrating from $0$ to $t$
  for some $0 < t \leq T$, and applying the Cauchy--Schwarz and
  Young's inequalities,
  \begin{multline*}
    \frac{\phi}{2} \norm[0]{c_h(t)}^2_{\Omega}
    + C\int_0^{t} \tnorm{\boldsymbol{c}_h(s)}_c^2 \dif s
    \\
    \leq \frac{\phi}{2} \norm[0]{c_h(0)}^2_{\Omega}
    + \frac{1}{2} \int_0^{t} \norm[0]{f(s)}_{\Omega}^2 \dif s 
    + \frac{1}{2} \int_0^t (1+ C\norm[0]{f^d}_{0,\infty,\Omega^d}) \norm[0]{c_h(s)}_{\Omega}^2 \dif s.
  \end{multline*}
  The result follows by \cref{lem:classicGronwall} and recalling that
  $c_h(0)$ is the $L^2$-projection of $c_0$ into $C_h$.  
\end{proof}

A consequence of this stability result is existence and uniqueness,
which can be obtained by setting $c_0 = 0$ and $f = 0$.
 
\begin{theorem}[existence and uniqueness]
  The semi-discrete scheme \cref{eq:semidiscrete} has a unique
  solution $\boldsymbol{c}_h\in \boldsymbol{C}_h$.
\end{theorem}

We next prove consistency of the method.

\begin{lemma}[consistency]
  \label{lem:consistency}
  If $c\in L^2(0,T;H^2(\Omega))$ solves \cref{eq:transport_system} and
  $u$ is the velocity solution to \cref{eq:system,eq:interface}, then
  for all $t\in (0,T]$,
  \begin{equation}
    \label{eq:consistency}
    \sum_{K\in \mathcal{T}} (\phi\, \partial_t c(t), w_h)_K + B_h(u; \boldsymbol{c}(t),\boldsymbol{w}_h)
    = \sum_{K\in \mathcal{T}}(f(t), w_h)_K \quad \forall \boldsymbol{w}_h \in \boldsymbol{C}_h,
  \end{equation}
  where $\boldsymbol{c}=(c, \bar{c})$ with $\bar{c}$ the restriction
  of $c$ to $\Gamma^0$.
\end{lemma}
\begin{proof}
  The result follows by integrating by parts, noting that
  $c = \bar{c}$ on $\partial K$, the boundary conditions
  \cref{eq:bc_s,eq:bc_d}, and
  \cref{eq:transport_eq,eq:transport_bc}.  
\end{proof}

\section{Error analysis of the semi-discrete scheme}
\label{sec:erroranalsemidiscrete}

Given the velocity solution $u_h \in V_h$ to \cref{eq:hdgwf} we prove
that the solution $c_h \in C_h$ to the semi-discrete scheme
\cref{eq:semidiscrete} converges to the solution of
\cref{eq:transport_system} in the energy norm and in the
$L^2$-norm. We consider only the analysis of a compatible
discretization for the general case when $f^d \ne 0$ in
\cref{eq:mass}, i.e., we take $k > 1$ in $\boldsymbol{X}_h$ and
$\ell = k-1 > 0$ in \cref{eq:spaces_transport}.

\subsection{Error estimate of the concentration in the energy norm}
To prove convergence in the energy norm we will use the continuous
interpolant $\mathcal{I}c\in C_h\cap \mathcal{C}^0(\bar{\Omega})$ of
$c$ \cite{Brenner:book} and we set
$\bar{\mathcal{I}} c(t) = \mathcal{I} c|_{\Gamma^0}(t)\in \bar{C}_h$.
Denoting the restriction of $c$ to $\Gamma^0$ by $\bar{c}$, we split
the approximation errors as follows:
\begin{equation*}
  c - c_h = \xi_c -  \zeta_c
  \quad\text{and}\quad
  \bar{c} - \bar{c}_h = \bar{\xi}_c - \bar{\zeta}_c,
\end{equation*}
where
\begin{gather*}
  \begin{aligned}
    \xi_c &= c - \mathcal{I} c, &\quad \zeta_c &= c_h
    - \mathcal{I} c, & \quad \boldsymbol{\xi}_c &=(\xi_c, \bar{\xi}_c),
    \\
    \bar{\xi}_c &= \bar{c} - \bar{\mathcal{I}} c, & \quad \bar{\zeta}_c
    &= \bar{c}_h - \bar{\mathcal{I}} c,
    &\quad \boldsymbol{\zeta}_c &=(\zeta_c, \bar{\zeta}_c).
  \end{aligned}
  \label{eq:split12}
\end{gather*}
The following interpolation estimates hold
\cite[Chapter 4]{Brenner:book}:
\begin{equation}
  \label{eq:interp0}
  \norm{\xi_c}_{r,K} \leq C h_K^{s-r} \norm{c}_{s,K}, \quad r=0, 1, \quad 2 \le s \le k,
\end{equation}
and
\begin{equation}
  \label{eq:interpinf}
  \norm{\xi_c}_{r,\infty,K} \leq C h_K^{s-r} \norm{c}_{s,\infty,K}, \quad r=0, 1, \quad 1 \le s \le k.
\end{equation}
A straightforward consequence is
\begin{equation}
  \label{eq:xi_cbound}
  \tnorm{\boldsymbol{\xi}_c(t)}_c
  = \del[2]{\sum_{K\in \mathcal{T}}\|\nabla \xi_c(t)\|_{K}^2}^{1/2}
  \leq C h^{s-1}\|c(t)\|_{s,\Omega}, \quad 2\leq s\leq k.
\end{equation}
We will also require the following continuity results.

\begin{lemma} 
  \label{lem:contresult1}
  Let $u\in \sbr[0]{W^{1,\infty}(\Omega)}^{\dim}$ be the velocity
  solution to \cref{eq:system,eq:interface} such that
  $u^s \in \sbr[0]{H^{k+1}(\Omega^s)}^{\dim}$ and
  $u^d \in \sbr[0]{H^{k+1}(\Omega^d)}^{\dim}$, let
  $c \in L^2(0,T; H^{k}(\Omega))$ be the solution to
  \cref{eq:transport_system} and let $u_h \in V_h$ be the velocity
  solution to \cref{eq:hdgwf}. Then for any $0\leq t\leq T$,
  \begin{equation}
    \label{eq:continuity-1}
    B_h(u_h; \boldsymbol{\xi}_c(t), \boldsymbol{w}_h)
    \leq Ch^{k-1}\norm[0]{c(t)}_{k,\Omega} \tnorm{\boldsymbol{w}_h}_c,
  \end{equation}
  for all $\boldsymbol{w}_h\in \boldsymbol{C}_h$ and $k > 1$. 
\end{lemma}
\begin{proof}
  By definition of $B_h$ \cref{eq:bilinearform_Bh},
  \begin{equation*}
    B_h(u_h;\boldsymbol{\xi}_c,\boldsymbol{w}_h)
    = B_h^a(u_h;\boldsymbol{\xi}_c,\boldsymbol{w}_h)
    + B_h^d(u_h;\boldsymbol{\xi}_c,\boldsymbol{w}_h).    
  \end{equation*}
  We will first bound $B_h^a$. Noting that $\xi_c-\bar{\xi}_c$
  vanishes on facets, we have by \cref{eq:bilinearform_Bha}
  \begin{equation*}
      B_h^a(u_h; \boldsymbol{\xi}_c, \boldsymbol{w}_h)
      = -\sum_{K\in \mathcal{T}}  (\xi_c , u_h \cdot \nabla w_h )_K
      + \sum_{K\in \mathcal{T}}\langle \xi_c \, u_h \cdot n , w_h-\bar{w}_h \rangle_{\partial K}
      =  J_1 + J_2.
  \end{equation*}
  We start by bounding $J_1$. By \cref{eq:interp0}, we have
  \begin{equation*}
      J_{1} \leq C \norm{u_h}_{0,\infty, \Omega} \norm{\xi_c}_{\Omega} 
      \norm{\nabla w_h}_{\Omega}
      \leq Ch^{k} \norm{u_h}_{0,\infty, \Omega} \norm{c}_{k,\Omega} 
      \tnorm{\boldsymbol{w}_h}_c.
  \end{equation*}
  By \cref{eq:2-trace-continuous,eq:inf-trace-continuous,eq:interp0},
  \begin{equation*}
    \begin{split}
      J_{2}
      & \leq C\norm{u_h}_{0,\infty, \Omega}
      \del[2]{\sum_{K\in \mathcal{T}}h_K\norm{\xi_c}_{\partial K}^2}^{1/2}
      \del[2]{\sum_{K\in \mathcal{T}}h_K^{-1}\norm{w_h-\bar{w}_h}_{\partial K}^2}^{1/2}
      \\
      & \leq Ch^{k} \norm{u_h}_{0,\infty, \Omega} \norm{c}_{k,\Omega}
      \tnorm{\boldsymbol{w}_h}_c.
    \end{split}
  \end{equation*}
  Combining the bounds for $J_1$ and $J_2$ and using \cref{eq:uhbound-Linf} we obtain
  \begin{equation}
    \label{eq:Bha-bound}
    B_h^a(u_h; \boldsymbol{\xi}_c, \boldsymbol{w}_h)
    \leq Ch^{k}\norm{c}_{k,\Omega}\tnorm{\boldsymbol{w}_h}_c.
  \end{equation}
  We next bound $B_h^d$. Since
  $u_h \in \sbr[0]{L^{\infty}(\Omega)}^{\dim}$ by
  \cref{eq:uhbound-Linf}, we can use \cref{eq:D_max} to bound
  $D(u_h)$. Then, noting that $\xi_c-\bar{\xi}_c$ vanishes
  on facets, and using \cref{eq:2-trace-continuous} and \cref{eq:interp0},
  \begin{equation}
    \label{eq:Bhd-bound}
    \begin{split}
      B_h^d(u_h;\boldsymbol{\xi}_c,\boldsymbol{w}_h)
      \leq & D_{\max}\norm{\nabla \xi_c}_{\Omega} \tnorm{\boldsymbol{w}_h}_c
      \\
      & + D_{\max}\del[2]{\sum_{K\in \mathcal{T}} h_K\norm{\nabla \xi_c}_{\partial K}^2}^{1/2}
      \del[2]{\sum_{K\in \mathcal{T}}h_K^{-1}\norm{w_h-\bar{w}_h}_{\partial K}^2}^{1/2}
      \\
      \leq & C h^{k-1} \norm{c}_{k,\Omega} \tnorm{\boldsymbol{w}_h}_{c}.
    \end{split}
  \end{equation}
  The result follows after combining
  \cref{eq:Bha-bound,eq:Bhd-bound}.  
\end{proof}

\begin{lemma}
  \label{lem:contresult2}
  Let $u\in \sbr[0]{W^{1,\infty}(\Omega)}^{\dim}$ be the velocity
  solution to \cref{eq:system,eq:interface} such that
  $u^s \in \sbr[0]{H^{k+1}(\Omega^s)}^{\dim}$ and
  $u^d \in \sbr[0]{H^{k+1}(\Omega^d)}^{\dim}$,
  $c \in L^2(0,T; H^{k}(\Omega)\cap W^{1,\infty}(\Omega))$ be the
  solution to \cref{eq:transport_system} and let $u_h \in V_h$ be the
  velocity solution to \cref{eq:hdgwf}. Then for any $0\leq t\leq T$,
  \begin{equation}
    \label{eq:continuity-2}
    B_h(u_h; \boldsymbol{c}(t), \boldsymbol{w}_h) - B_h(u; \boldsymbol{c}(t), \boldsymbol{w}_h)
    \leq
    Ch^{k-1} \del[1]{\norm[0]{c(t)}_{1,\infty,\Omega} + \norm[0]{c(t)}_{k,\Omega}}\tnorm{\boldsymbol{w}_h}_{c},
  \end{equation}
  for all $\boldsymbol{w}_h\in \boldsymbol{C}_h$, where 
  $\boldsymbol{c}=(c, \bar{c})$ and $k > 1$.
\end{lemma}
\begin{proof}  
  We first note that
  \begin{equation*}
    B_h(u_h; \boldsymbol{c} ,\boldsymbol{w}_h) - B_h(u; \boldsymbol{c}, \boldsymbol{w}_h)
    =
    B_h^a(u_h-u; \boldsymbol{c}, \boldsymbol{w}_h)
    +
    [B_h^d(u_h; \boldsymbol{c}, \boldsymbol{w}_h) - B_h^d(u; \boldsymbol{c}, \boldsymbol{w}_h)].
  \end{equation*}
  Since $c = \bar{c}$ on element boundaries,
  \begin{multline*}  
    B_h^a(u_h-u; \boldsymbol{c}, \boldsymbol{w}_h)
    = -\sum_{K\in\mathcal{T}_h}(c (u_h-u),\nabla w_h)_K 
    + \sum_{K\in \mathcal{T}_h}\langle c (u_h-u)\cdot n, w_h -\bar{w}_h)_{\partial K} 
    \\
    = G_1 + G_2. 
  \end{multline*}
  By \cref{eq:L2error-u},
  \begin{equation*}
    G_1 \leq \norm{c}_{0,\infty, \Omega} \norm{u-u_h}_{\Omega}\norm{\nabla w_h}_{\Omega} 
    \leq C h^{k+1} \norm{c}_{0,\infty, \Omega}\tnorm{\boldsymbol{w}_h}_c.
  \end{equation*}
  Furthermore, by \cref{eq:inf-trace-continuous,eq:bounduhuboundary},
  \begin{equation*}
    \begin{split}
      G_2
      &\leq C \norm{c}_{0,\infty,\Omega}
      \del[2]{\sum_{K\in \mathcal{T}_h}h_K \norm{u-u_h}^2_{\partial K}}^{1/2} \del[2]{\sum_{K\in \mathcal{T}_h}h_K^{-1} \norm{w_h-\bar{w}_h}^2_{\partial K}}^{1/2}
      \\
      &\leq Ch^{k+1} \norm{c}_{0,\infty, \Omega} \tnorm{\boldsymbol{w}_h}_c.      
    \end{split}    
  \end{equation*}
  It follows that
  $ B_h^a(u_h-u; \boldsymbol{c}, \boldsymbol{w}_h) \leq Ch^{k+1}
  \norm{c}_{0,\infty, \Omega} \tnorm{\boldsymbol{w}_h}_c$.
  
  Consider now
  $B_h^d(u_h; \boldsymbol{c}, \boldsymbol{w}_h) - B_h^d(u;
  \boldsymbol{c}, \boldsymbol{w}_h)$. Since $c = \bar{c}$ on
  $\partial K$,
  \begin{equation*}
    \begin{split}
      B_h^d(u_h; \boldsymbol{c},\boldsymbol{w}_h)
      - B_h^d(u; \boldsymbol{c},\boldsymbol{w}_h)
      =& \sum_{K\in \mathcal{T}} ( (D(u_h)-D(u))\nabla c, \nabla w_h )_K
      \\
      &- \sum_{K\in \mathcal{T}} 
      \langle [(D(u_h)-D(u))\nabla c] \cdot n , w_h-\bar{w}_h\rangle_{\partial K}
      \\
      =& H_1 + H_2.      
    \end{split}
  \end{equation*}  
  Using \cref{eq:DLipschitz,eq:L2error-u},
  \begin{equation*}
      H_1
      \leq C\norm{u_h-u}_{\Omega} \norm{\nabla c}_{0,\infty, \Omega}\norm{\nabla w_h}_{\Omega}
      \leq C h^{k+1} \norm{c}_{1,\infty,\Omega} \tnorm{\boldsymbol{w}_h}_c.
  \end{equation*}
  To bound $H_2$, we first rewrite it as follows:
  \begin{align*}
    H_2
    =&  \sum_{K\in \mathcal{T}}\langle [(D(u_h)-D(u))\nabla \mathcal{I} c] \cdot n , w_h-\bar{w}_h\rangle_{\partial K}
    \\
    &+ \sum_{K\in \mathcal{T}}\langle [D(u)\nabla (\mathcal{I} c - c)] \cdot n , w_h-\bar{w}_h\rangle_{\partial K}
    \\
    &- \sum_{K\in \mathcal{T}}\langle [D(u_h)\nabla (\mathcal{I} c - c)] \cdot n , w_h-\bar{w}_h\rangle_{\partial K}
    \\
    =&  H_{21}+H_{22}+H_{23}.
  \end{align*}
  By \cref{eq:bounduhuboundary},
  \cref{eq:DLipschitz}, and \cref{eq:interpinf},
  \begin{align*}
    H_{21}
    &\leq \sum_{K\in \mathcal{T}}C\norm[0]{u_h-u}_{\partial K}\norm[0]{\nabla \mathcal{I} c}_{0,\infty,K}\norm[0]{w_h-\bar{w}_h}_{\partial K}
    \\
    &\leq C \norm[0]{c}_{1,\infty,\Omega}\del[2]{\sum_{K\in \mathcal{T}}h_K\norm[0]{u_h-u}^2_{\partial K}}^{1/2}
      \del[2]{\sum_{K\in \mathcal{T}}h_K^{-1}\norm[0]{w_h-\bar{w}_h}^2_{\partial K}}^{1/2}
    \\
    & \leq Ch^{k+1}\norm[0]{c}_{1,\infty,\Omega}\tnorm{\boldsymbol{w}_h}_c.
  \end{align*}
  Using \cref{eq:Dboundary}, \cref{eq:2-trace-continuous}, and
  \cref{eq:interp0}, we obtain
  \begin{align*}
    H_{22}
    &\leq C(1 + \norm[0]{u}_{0,\infty,\Omega}) \del[2]
      {\sum_{K\in \mathcal{T}}h_K\norm{\nabla (\mathcal{I} c - c)}_{\partial K}^2}^{1/2}	\del[2]
      {\sum_{K\in \mathcal{T}}h_K^{-1}\norm[0]{w_h-\bar{w}_h}^2_{\partial K}}^{1/2}
    \\
    &\leq C(1 + \norm[0]{u}_{0,\infty,\Omega}) h^{k-1}\norm[0]{c}_{k,\Omega}\tnorm{\boldsymbol{w}_h}_c.
  \end{align*}
  Following similar steps as above, this time using
  \cref{eq:inf-trace-continuous,eq:D_max}, we obtain
  \begin{equation*}
    H_{23}
    \leq D_{\max}h^{k-1}\norm[0]{c}_{k,\Omega}\tnorm{\boldsymbol{w}_h}_c.
  \end{equation*}
  The result follows.  
\end{proof}
We next find an estimate of the concentration approximation error in
the energy norm.
\begin{lemma}
  \label{lem:energynormerror}
  Let $u\in \sbr[0]{W^{1,\infty}(\Omega)}^{\dim}$ be the velocity
  solution to \cref{eq:system,eq:interface} such that
  $u^s \in \sbr[0]{H^{k+1}(\Omega^s)}^{\dim}$ and
  $u^d \in \sbr[0]{H^{k+1}(\Omega^d)}^{\dim}$ with $k > 1$, and let
  $c \in L^2(0,T; H^{k}(\Omega)\cap W^{1,\infty}(\Omega))$ be the
  solution to \cref{eq:transport_system} such that
  $\partial_t c\in L^2(0, T; H^{k}(\Omega))$ and
  $c_0\in H^{k}(\Omega)$. Furthermore, let $u_h \in V_h$ be the
  velocity solution to \cref{eq:hdgwf}. Then,
  \begin{equation*}
    \label{eq:energynorm-simplified}
    \norm[0]{\zeta_c(t)}^2_{\Omega}
    + \int_0^{t} \tnorm{\boldsymbol{\zeta}_c(s)}_c^2 \dif s
    \leq Ch^{2(k-1)} \quad \forall t\in[0,T].
  \end{equation*}
\end{lemma}
\begin{proof}
  Subtracting \cref{eq:consistency} from
  \cref{eq:semidiscrete} and splitting the error,
  \begin{multline*}
    \sum_{K\in \mathcal{T}}(\phi \partial_t \zeta_c, w_h)_K 
    + B_h(u_h; \boldsymbol{\zeta}_c, \boldsymbol{w}_h) 
    =
    \sum_{K\in \mathcal{T}}(\phi \partial_t \xi_c, w_h)_K 
    \\
    + B_h(u_h; \boldsymbol{\xi}_c, \boldsymbol{w}_h) 
    -  \sbr[1]{B_h(u_h; \boldsymbol{c}, \boldsymbol{w}_h)
    - B_h(u;\boldsymbol{c}, \boldsymbol{w}_h)}.
  \end{multline*}
  Setting $\boldsymbol{w}_h = \boldsymbol{\zeta}_c(t)$, using
  coercivity \cref{eq:coercivity}, and integrating from $0$ to $t$,
  where $0\leq t\leq T$, we obtain
  \begin{multline*}
    \label{eq:splittingbounds}
      \frac{\phi}{2}\norm[0]{\zeta_c(t)}^2_{\Omega} 
      + C\int_0^{t}\tnorm{\boldsymbol{\zeta}_c(s)}_c^2\dif s 
      \leq  \frac{\phi}{2}\norm[0]{\zeta_c(0)}^2_{\Omega} 
      + \int_0^{t}\sum_{K\in \mathcal{T}}(\phi \partial_t \xi_c(s), \zeta_c(s))_K \dif s
      \\
      + \int_0^t B_h(u_h;\boldsymbol{\xi}_c(s), \boldsymbol{\zeta}_c(s)) \dif s
      - \int_0^t B_h(u_h;\boldsymbol{c}(s), \boldsymbol{\zeta}_c(s)) \dif s    
      + \int_0^{t}B_h(u;\boldsymbol{c}(s), \boldsymbol{\zeta}_c(s))\dif s \\
      - \int_0^{t} \tfrac12\sum_{K\in \mathcal{T}^d} (\nabla \cdot u_h, \zeta_c^2(s))_K \dif s 
      =: I_1+\hdots+I_6.
  \end{multline*}
  Consider first $I_1$. Observe that by \cref{eq:interp0}
  \begin{equation*}
    \norm[0]{\zeta_c(0)}_{\Omega}
    \leq \norm[0]{c_h(0)-c_0}_{\Omega} + \norm[0]{c_0-\mathcal{I} c_0}_{\Omega}
    \leq Ch^{k}\norm[0]{c_0}_{k,\Omega},    
  \end{equation*}
  since $c_h(0)$ is the $L^2$-projection of $c_0$ into
  $C_h$. Therefore
  $I_1 \leq C\frac{\phi}2h^{2(k)}\norm{c_0}^2_{k,\Omega}$.

  Using \cref{eq:interp0} and Cauchy--Schwarz and Young's
  inequalities,
  \begin{equation*}
    I_2 \le  \frac{C}2 h^{2k} \norm[0]{\partial_t c}_{L^2(0,t;H^{k}(\Omega))}^2 
    + \frac{C}2 \int_0^t\norm[0]{\zeta_c(s)}_{\Omega}^2\dif s.    
  \end{equation*}
  Next, by \cref{eq:continuity-1} and Young's inequality,
  \begin{equation*}
    I_3 
    \le C h^{2(k-1)}\norm{c}_{L^2(0,t;H^k(\Omega))}^2
    + \delta \int_0^t \tnorm{\boldsymbol{\zeta}_c(s)}_c^2 \dif s,     
  \end{equation*}
  with $\delta>0$ a constant. Similarly, using \cref{eq:continuity-2},
  \begin{equation*}
    I_4 + I_5
    \leq  Ch^{2(k-1)} \del[1]{ \norm{c}^2_{L^2(0,t;W^{1,\infty}(\Omega))} + \norm{c}^2_{L^2(0,t;H^k(\Omega))} }
    + \delta \int_0^{t} \tnorm{\boldsymbol{\zeta}_c(s)}_c^2 \dif s.
  \end{equation*}  
  For the final term, as in the proof of \Cref{lem:stability}, we have
  \begin{equation*}
    I_6\le \int_0^t C \norm[0]{f^d}_{0,\infty,\Omega^d}\norm[0]{\zeta_c(s)}_{\Omega}^2 \dif s.    
  \end{equation*}
  Choosing $\delta$ small enough and combining the above bounds, we
  obtain
  \begin{equation*}
    \begin{split}
      &\norm[0]{\zeta_c(t)}^2_{\Omega} + \int_0^{t} \tnorm{\boldsymbol{\zeta}_c(s)}_c^2 \dif s
      \\
      &\le C h^{2(k-1)}\del{
        \norm[0]{c_0}^2_{k,\Omega}       
        + \norm[0]{\partial_t c}_{L^2(0,t;H^{k}(\Omega))}^2 
        + \norm[0]{c}_{L^2(0,t;H^k(\Omega))}^2       
        + \norm[0]{c}_{L^2(0,t;W^{1,\infty}(\Omega))}^2}
      \\
      & \qquad +\int_0^t C\del[1]{ 1+ \norm[0]{f^d}_{0,\infty,\Omega^d}} \norm[0]{\zeta_c(s)}_{\Omega}^2 \dif s.      
    \end{split}
  \end{equation*}
  The result follows by \cref{lem:classicGronwall}.  
\end{proof}

A consequence of \cref{eq:interp0,eq:xi_cbound} and
\cref{lem:energynormerror} is the following error estimate for the
concentration in the energy norm.

\begin{corollary}
  \label{cor:energynormerror}
  Let $u$, $c$ and $c_0$ satisfy the assumptions of
  \cref{lem:energynormerror}. Then,
  \begin{equation*}
    \label{eq:energynorm-corollary}
    \norm[0]{c(t)-c_h(t)}^2_{\Omega}
    + \int_0^{t} \tnorm{\boldsymbol{c}(s) - \boldsymbol{c}_h(s)}_c^2 \dif s
    \le C h^{2(k-1)} \quad \forall t \in [0, T].
  \end{equation*}
\end{corollary}

\subsection{Error estimate of the concentration in the $L^2$-norm}
To obtain the $L^2$-norm estimate for the concentration we consider
the following dual problem for $0 < t < T$ \cite{Dawson:2001}:
\begin{subequations}
  \begin{align}
    \label{eq:dual1}
    \phi\partial_t \sigma + u\cdot \nabla \sigma + \nabla \cdot [D(u) \nabla \sigma]& =\Psi & \quad & \mbox{ in }\Omega,
    \\
    \label{eq:dual2}
    D(u)\nabla \sigma\cdot n & = 0  & \quad & \mbox{ on }\partial \Omega,
    \\
    \label{eq:dual3}
    \sigma(\cdot, T)  & =0 & \quad & \mbox{ in }\Omega.
  \end{align}
  \label{eq:dual}
\end{subequations}
We will assume that $\sigma\in W^{2,\infty}(\Omega)$ and that
\begin{equation}
  \label{eq:regularity}
  \max_{0\leq t\leq T} \norm[0]{\sigma(\cdot,t)}^2_{\Omega}
  + \int_0^T \norm[0]{\sigma(t)}_{2,\Omega}^2 \dif t
  \leq C \norm[0]{\Psi}_{L^2(0,T;L^2(\Omega))}^2.
\end{equation}
For the analysis in this section we require $\Pi_C$ and $\bar{\Pi}_C$,
the $L^2$-projections into $C_h$ and $\bar{C}_h$, respectively. We
will use the following standard results of approximation theory
\cite[Chapter 3]{Ciarlet:book}:
\begin{align}
  \label{eq:L2projerror-1}
  \norm[0]{w-\Pi_Cw}_{r,K}  &\le Ch^{s-r}\norm[0]{w}_{s,K}, 
  && 0 \leq s\leq k, && r=0,1, 
  \\
  \label{eq:L2projerror-2}
  \norm[0]{\Pi_Cw-\bar{\Pi}_Cw}_{\partial K} &\le Ch^{s-1/2}\norm[0]{w}_{s,K}, 
  && 1\leq s\leq k,
\end{align}
for all $w \in H^k(\Omega)$ .
We will now prove an $L^2$-norm estimate for the concentration.

\begin{theorem}[Concentration error estimate in the $L^2$-norm.]
  \label{thm:l2errorestc}
  Suppose that $c$, $u$ and $u_h$ satisfy the assumptions of
  \Cref{lem:energynormerror} and $c_h$ is the solution to
  \cref{eq:semidiscrete}. Then
  \begin{equation*}
    \norm[0]{c-c_h}_{L^2(0,T;L^2(\Omega))} \leq C h^{k}.
  \end{equation*}
\end{theorem}
\begin{proof}
  Set $\Psi = c-c_h$ in \cref{eq:dual}. We multiply \cref{eq:dual1} by $c-c_h$,
  integrate over $K \in \mathcal{T}$, integrate by parts the diffusion
  term over an element $K \in \mathcal{T}$, sum over the elements, and
  integrate from 0 to $T$. Then, integrating by parts in time and
  using \cref{eq:dual3},
  \begin{multline}
    \label{eq:L2error-1}
    -\int_0^T \sum_{K\in \mathcal{T}}(\phi  \sigma, \partial_t(c-c_h))_K 
    +\int_0^T\sum_{K\in \mathcal{T}}(\nabla \sigma,u(c-c_h))_K\dif t
    \\
    -\int_0^T\sum_{K\in \mathcal{T}}(D(u)\nabla \sigma, \nabla (c-c_h))_K\dif t 
    +\int_0^T\sum_{K\in \mathcal{T}}\langle D(u)\nabla \sigma\cdot n, c-c_h\rangle_{\partial K}\dif t
    \\
    - \sum_{K\in \mathcal{T}}(\phi  \sigma(0), c_0-c_h(0))_K 
    =  \norm[0]{c-c_h}_{L^2(0,T;L^2(\Omega))}^2.
  \end{multline}
  Subtracting \cref{eq:semidiscrete} from \cref{eq:consistency},
  choosing
  $\boldsymbol{w}_h = \boldsymbol{\Pi}_C\sigma = (\Pi_C \sigma,
  \bar{\Pi}_C \sigma)$ and integrating with respect to time,
  \begin{equation}
    \label{eq:L2error-2}
    \int_0^T\sum_{K\in \mathcal{T}} \cbr{
      (\phi\, \partial_t (c-c_h)\,, \Pi_C \sigma)_K
      + \sbr{B_h(u; \boldsymbol{c}, \boldsymbol{\Pi}_C\sigma)
      - B_h(u_h; \boldsymbol{c}_h, \boldsymbol{\Pi}_C\sigma)} } \dif t
      = 0.    
  \end{equation}
  Adding \cref{eq:L2error-2,eq:L2error-1}, and manipulating
  the integrals, we obtain
  \begin{equation*}
    \begin{split}
       & \norm{c-c_h}_{L^2(0,T;L^2(\Omega))}^2 
      = -\int_0^T\sum_{K\in \mathcal{T}} (\phi\, \partial_t (c-c_h)\,, \sigma-\Pi_C\sigma)_K \dif t
      \\
      & + \int_0^T\sum_{K\in \mathcal{T}}(\nabla (\sigma-\Pi_C\sigma),u(c-c_h))_K\dif t
      -\int_0^T\sum_{K\in \mathcal{T}}(D(u)\nabla (\sigma-\Pi_C\sigma), \nabla (c-c_h))_K\dif t
      \\
      & + \int_0^T\sum_{K\in \mathcal{T}}\langle [D(u)\nabla \sigma]\cdot n, c-c_h\rangle_{\partial K}\dif t
      + \int_0^T \sum_{K\in \mathcal{T}} \langle [D(u_h)\nabla \Pi_C \sigma]\cdot n, c_h-\bar{c}_h\rangle_{\partial K} \dif t
      \\
      & - \int_0^T \sum_{K\in \mathcal{T}} (c_h(u-u_h),\nabla \Pi_C \sigma)_K \dif t
      + \int_0^T \sum_{K\in \mathcal{T}} (\nabla c_h, [D(u)-D(u_h)]\nabla \Pi_C \sigma)_K \dif t
      \\
      & - \sum_{K\in \mathcal{T}}(\phi  \sigma(0), c_0-c_h(0))_K 
      + \int_0^T\sum_{K\in \mathcal{T}} \langle (c-c_h)u\cdot n, \Pi_C\sigma-\bar{\Pi}_C\sigma\rangle_{\partial K}\dif t
      \\
      & + \int_0^T\sum_{K\in \mathcal{T}} \langle c_h(u-u_h)\cdot n, \Pi_C\sigma-\bar{\Pi}_C\sigma \rangle_{\partial K} \dif t
      \\
      & - \int_0^T\sum_{K\in \mathcal{T}} \langle [D(u)\nabla (c-c_h)]\cdot n, \Pi_C\sigma-\bar{\Pi}_C\sigma\rangle_{\partial K} \dif t
      \\
      & - \int_0^T\sum_{K\in \mathcal{T}} \langle [(D(u)-D(u_h))\nabla c_h]\cdot n,\Pi_C\sigma-\bar{\Pi}_C\sigma\rangle_{\partial K} \dif t
      \\
      & + \int_0^T\sum_{K\in \mathcal{T}} \langle u_h\cdot n(c_h-\bar{c}_h), \Pi_C\sigma-\bar{\Pi}_C\sigma \rangle_{\partial K^{\rm in}} \dif t
      \\
      & - \int_0^T\sum_{K\in \mathcal{T}} \frac{\beta_c}{h_K}\langle [D(u_h)n](c_h-\bar{c}_h), (\Pi_C\sigma-\bar{\Pi}_C\sigma)n \rangle_{\partial K} \dif t
      \\
      =:& T_1+\hdots + T_{14}.      
    \end{split}
  \end{equation*}
  Observe that since $c_h, \Pi_C c\in C_h$,
  $(\phi\, \partial_t c_h\,, \sigma-\Pi_C\sigma)_K = 0$ and
  $(\phi\, \partial_t \Pi_Cc\,, \sigma-\Pi_C\sigma)_K = 0$. Then, by
  \cref{eq:L2projerror-1}, the Cauchy--Schwarz inequality,
  \cref{eq:regularity}, and Young's inequality,
  \begin{equation*}
    \begin{split}
      T_1
      &= -\int_0^T\sum_{K\in \mathcal{T}} (\phi\, \partial_t (c-\Pi_Cc)\,, \sigma-\Pi_C\sigma)_K \dif t
      \\
      &\leq \int_0^T\sum_{K\in \mathcal{T}} C\phi h^{k}\norm{\partial_t c}_{k-2, K} \norm{\sigma}_{2, K}\dif t
      \\
      & \leq C\phi h^k\norm{\partial_t c}_{L^2(0,T;H^{k-2}(\Omega))} \del{ \int_0^T \norm{\sigma(t)}_{2,\Omega}^2 \dif t }^{1/2}
      \\
      & \leq Ch^{2k}\norm{\partial_t c}^2_{L^2(0,T;H^{k-2}(\Omega))} + \delta \norm{c-c_h}_{L^2(0,T;L^2(\Omega))}^2,
    \end{split}
  \end{equation*}
  where $\delta > 0$.

  By \cref{cor:energynormerror}, \cref{eq:regularity,eq:L2projerror-1},
  \begin{equation*}
    \begin{split}
      T_2
      &\leq \norm{u}_{0,\infty, \Omega} \del[2]{\int_0^T \norm{c-c_h}_{\Omega}\dif t}^{1/2} \del[2]{\int_0^T \norm{\nabla (\sigma-\Pi_C\sigma)}^2_{\Omega}\dif t}^{1/2}
      \\
      & \leq\norm{u}_{0,\infty, \Omega}\del[2]{\int_0^T Ch^{2(k-1)}\dif t}^{1/2} h\del[2]{\int_0^T \norm{\sigma}^2_{2,\Omega}\dif t}^{1/2}
      \\
      & \leq Ch^{k} T^{1/2}\norm{u}_{0,\infty, \Omega}\norm{c-c_h}_{L^2(0,T;L^2(\Omega))}
      \\
      & \leq \delta \norm{c-c_h}_{L^2(0,T;L^2(\Omega))}^2 + CT \norm{u}_{0,\infty, \Omega}^2h^{2k}.
    \end{split}
  \end{equation*}
  Following similar steps as in the bound for $T_2$, using
  \cref{eq:D_max} and \cref{cor:energynormerror}, we find that
  \begin{equation*}
    T_3 \le \delta \norm{c-c_h}_{L^2(0,T;L^2(\Omega))}^2 + CD_{max}^2h^{2k}.
  \end{equation*}
  Next, observe that by smoothness of $u$, $c$ and $\sigma$,
  \cref{eq:dual2} and single valuedness of $\bar{c}_h$, we can write
  \begin{equation*}
    \sum_{K\in \mathcal{T}} \langle D(u)\nabla \sigma\cdot n, c-c_h\rangle_{\partial K}
    = \sum_{K\in \mathcal{T}}\langle D(u)\nabla \sigma\cdot n, \bar{c}_h-c_h\rangle_{\partial K}.
  \end{equation*}
  Therefore, 
  \begin{multline*}
    T_4+T_5
    =
    \int_0^T\sum_{K\in \mathcal{T}}\langle D(u)\nabla (\sigma - \Pi_C \sigma) \cdot n, \bar{c}_h-c_h\rangle_{\partial K}\dif t
    \\
    + \int_0^T\sum_{K\in \mathcal{T}} \langle [(D(u_h)-D(u))\nabla \Pi_C \sigma]\cdot n, c_h-\bar{c}_h\rangle_{\partial K} \dif t
    = T_{451}+T_{452}.
  \end{multline*}
  Since $\bar{c}_h-c_h = (c-c_h) - (\bar{c}-\bar{c}_h)$ on
  $\partial K$, using
  \cref{eq:2-trace-continuous,eq:Dboundary,eq:L2projerror-1},
  \begin{equation*}
    \begin{split}
      T_{451}
      & \leq C (1+\|u\|_{0,\infty,\Omega}) \int_0^T\del[2]{\sum_{K\in \mathcal{T}}
      h_K\norm{\nabla (\sigma-\Pi_C\sigma)}_{\partial K}^2}^{1/2}\del[2]{\sum_{K\in \mathcal{T}}h_K^{-1}\norm{\bar{c}_h-c_h}^2_{\partial K}}^{1/2}\dif t
      \\
      &\leq Ch(1+\|u\|_{0,\infty,\Omega}) \del[2]{\int_0^T\norm{\sigma}_{2, \Omega}^2 \dif t}^{1/2}
      \del[2]{\int_0^T  \tnorm{\boldsymbol{c}-\boldsymbol{c}_h}_c^2 \dif t}^{1/2}.
    \end{split}
  \end{equation*}
  Next, apply \cref{eq:DLipschitz}, the trace inequality
  \cref{eq:inf-trace-continuous}, \cref{eq:bounduhuboundary}, the
  inverse inequality \cref{eq:inverse}, and \cref{eq:L2projerror-1},  
  \begin{equation*}
    \begin{split}
      T_{452}
      & \leq Ch^{k+1/2}
      \int_0^T \sum_{K\in \mathcal{T}}\norm{\nabla \Pi_C \sigma}_{0,\infty, K} \norm{c_h-\bar{c}_h}_{\partial K}\dif t
      \\
      & \leq C h^{k+1 - \dim/2}
      \int_0^T\sum_{K\in \mathcal{T}}\norm{\nabla \Pi_C \sigma}_{K}h_K^{-1/2}\norm{c_h-\bar{c}_h}_{\partial K}\dif t
      \\
      &\leq Ch\del[2]{\int_0^T|\sigma|_{1,\Omega}^2 \dif t}^{1/2}\del[2]{\int_0^T \tnorm{\boldsymbol{c}-\boldsymbol{c}_h}_c^2\dif t}^{1/2}.
    \end{split}
  \end{equation*}
  Therefore, by \cref{eq:regularity} and \cref{cor:energynormerror},
  \begin{equation*}
      T_4+T_5
      \leq Ch^k \norm{c-c_h}^2_{L^2(0,T;L^2(\Omega))}
      \leq \delta \norm{c-c_h}^2_{L^2(0,T;L^2(\Omega))} + Ch^{2k}.
  \end{equation*}
  To bound $T_6$, note that
  \begin{equation*}
    \begin{split}
      T_6 &= \int_0^T \sum_{K\in \mathcal{T}} ((c-c_h)(u-u_h),\nabla \Pi_C \sigma)_K \dif t
      - \int_0^T \sum_{K\in \mathcal{T}} (c(u-u_h),\nabla \Pi_C \sigma)_K \dif t
      \\
      & =: T_{61} + T_{62}.      
    \end{split}
  \end{equation*}
  Using the inverse inequality \cref{eq:inverse} and
  \cref{eq:L2error-u},
  \begin{equation*}
    \begin{split}
      T_{61}
      & \leq  \int_0^T \sum_{K\in \mathcal{T}} \norm{c_h-c}_K \norm{u-u_h}_K
      \norm{\nabla \Pi_C \sigma}_{0,\infty, K} \dif t
      \\
      & \leq C h^{k+1-{\rm \dim}/2}\int_0^T \norm{c_h-c}_{\Omega} \norm{\nabla \sigma}_{\Omega} \dif t
      \leq C h^{k}T^{1/2}\del[2]{\int_0^T  \norm{\nabla \sigma}^2_{\Omega} \dif t}^{1/2},
    \end{split}
  \end{equation*}
  where we used \cref{cor:energynormerror} and the assumption that
  $k > 1$.  Therefore, by \cref{eq:regularity} and Young's inequality,
  \begin{equation*}
      T_{61}  \leq C h^k T^{1/2} \norm{c - c_h}_{L^2(0,T;L^2(\Omega))}
      \le C h^{2k} + \delta\norm{c - c_h}^2_{L^2(0,T;L^2(\Omega))}.
  \end{equation*}
  By \cref{eq:L2projerror-1}, H\"older's inequality and
  \cref{eq:L2error-u},
  \begin{equation*}
    \begin{split}
      T_{62}
      &
      \leq C\norm{u-u_h}_{\Omega}\int_0^T \norm{c}_{0,\infty, \Omega}  \norm{\sigma}_{1,\Omega} \dif t
      \\
      &
      \leq Ch^{k+1}\norm{c}_{L^2(0,T;L^{\infty}(\Omega))}\del[2]{\int_0^T\norm{\sigma}_{2,\Omega}^2 \dif t}^{1/2}
      \leq  C h^{2k+2} + \delta \norm{c-c_h}^2_{L^2(0,T;L^2(\Omega))},
      \end{split}
  \end{equation*}
  where we applied \cref{eq:regularity}, and Young's
  inequality.  Therefore,
  \begin{equation*}
    T_6\leq C h^{2k} + \delta \norm{c-c_h}^2_{L^2(0,T;L^2(\Omega))}.
  \end{equation*}
  Similarly, this time using \cref{eq:DLipschitz},
  $T_7 \leq Ch^{2k} + \delta \norm{c-c_h}^2_{L^2(0,T;L^2(\Omega))}$.
  By \cref{eq:regularity}, Young's inequality and since $c_h(0)$ is
  the $L^2$-projection of $c_0$,
  \begin{equation*}
    T_8
    \leq \phi \norm[0]{\sigma(0)}_{\Omega} \norm[0]{c_0-c_h(0)}_{\Omega}
    \leq \delta \norm[0]{c-c_h}_{L^2(0,T;L^2(\Omega))}^2 + Ch^{2k} \norm[0]{c_0}^2_{k,\Omega}.    
  \end{equation*}
  Using \cref{eq:L2projerror-2}, \cref{eq:2-trace-continuous},
  \cref{eq:inf-trace-continuous}, \cref{cor:energynormerror}, and
  \cref{eq:regularity},
  \begin{equation*}
    \begin{split}
      T_9
      & \leq \norm{u}_{0,\infty, \Omega}\del[2]{\int_0^T  \sum_{K\in \mathcal{T}}\norm{c-c_h}_{\partial K}^2}^{1/2}\del[2]{\int_0^T  \sum_{K\in \mathcal{T}}\norm{\Pi_C\sigma - \bar{\Pi}_C\sigma}_{\partial K}^2}^{1/2}
      \\
      & \leq Ch^{3/2}\norm{u}_{0,\infty, \Omega}\del[2]{\int_0^T \sum_{K\in \mathcal{T}}\del[1]{h_K^{-1}\norm{c-c_h}_{K}^2 + h_K|c-c_h|_{1,K}^2}}^{1/2}
       \del[2]{\int_0^T\norm{\sigma}_{2,\Omega}^2 \dif t}^{1/2}
      \\
      & \leq Ch^{3/2}\norm{u}_{0,\infty, \Omega}  \del[2]{ T^{1/2} h^{k-3/2} + h^{k-1/2}}
       \del[2]{\int_0^T\norm{\sigma}_{2,\Omega}^2 \dif t}^{1/2}
      \\
      & \leq  Ch^{2k} + \delta \norm{c-c_h}^2_{L^2(0,T;L^2(\Omega))}.
    \end{split}
  \end{equation*}
  By \cref{eq:inf-trace-continuous}, \cref{eq:bounduhuboundary},
  \cref{eq:inverse}, \cref{eq:L2projerror-2}, H\"older's inequality,
  \cref{lem:stability}, that $k+2-{\rm dim}/2 \ge k $ and
  \cref{eq:regularity},
  \begin{equation*}
    \begin{split}
      T_{10}
      &\leq Ch^{k+1/2} \int_0^T\sum_{K\in \mathcal{T}} \norm{c_h}_{0,\infty, K}  \norm{\Pi_C\sigma-\bar{\Pi}_C\sigma}_{\partial K} \dif t
      \\
      &\leq Ch^{k+2-\rm{\dim}/2} \int_0^T\norm{c_h}_{\Omega}\norm{\sigma}_{2,\Omega} \dif t
      \\
      &\leq C h^{k+2-{\rm dim}/2} \max_{0\leq t\leq T} \norm{c_h}_{\Omega} T^{1/2}\|c-c_h\|_{L^2(0,T;L^2(\Omega))}\\
      &\le C h^{2k} + \delta \norm{c-c_h}^2_{L^2(0,T;L^2(\Omega))}.
    \end{split}
  \end{equation*}
  Using \cref{eq:Dboundary} and \cref{eq:L2projerror-2},
  \begin{equation*}
    T_{11}
    \leq C(1+\|u\|_{0,\infty,\Omega}) \int_0^T\sum_{K\in \mathcal{T}}
    \del{\norm{\nabla\xi_c}_{\partial K} + \norm{\nabla\zeta_c}_{\partial K}}
    h_K^{3/2}\norm{\sigma}_{2,K} \dif t.
  \end{equation*}
  By the trace inequalities \cref{eq:2-trace-continuous,eq:2-trace},
  the interpolation estimate \cref{eq:interp0}, and the
  Cauchy--Schwarz and Young's inequalities,
  \begin{equation*}
    \begin{split}
      T_{11}
      &\leq C(1+\|u\|_{0,\infty,\Omega})
      \sbr[2]{h^{k+1/2} \norm{c}_{L^2(0,T;H^k(\Omega))} + h \del[2]{\int_0^T\tnorm{\zeta_c}_c^2\dif t}^{1/2}}
      \del[2]{\int_0^T\norm{\sigma}_{2,\Omega}\dif t}^{1/2}
      \\
      & \leq Ch^{2k} + \delta \norm{c-c_h}^2_{L^2(0,T;L^2(\Omega))}.
    \end{split}
  \end{equation*}
  By \cref{eq:DLipschitz}, \cref{eq:bounduhuboundary},
  \cref{eq:L2projerror-2}, \cref{eq:inf-trace-continuous},
  \cref{eq:inverse}, \cref{lem:stability}
  and
  \cref{eq:regularity},
  \begin{multline*}
    T_{12}
    \leq C\int_0^T\sum_{K\in \mathcal{T}} \norm{u-u_h}_{\partial K}
    \norm{\nabla c_h}_{0,\infty, K}h_K^{3/2}\norm{\sigma}_{2, K} \dif t
    \\
    \leq  Ch^{k+2-\dim/2}\int_0^T
    \norm{\nabla c_h}_{\Omega} \norm{\sigma}_{2, \Omega} \dif t
    \leq  Ch^{2k} + \delta \norm{c-c_h}^2_{L^2(0,T;L^2(\Omega))}.
  \end{multline*}
  Using \cref{eq:inf-trace-continuous}, \cref{eq:L2projerror-2},
  \cref{cor:energynormerror}, \cref{eq:regularity} and
  \cref{eq:uhbound-Linf},
  \begin{align*}
    T_{13}
    & \leq C\|u_h\|_{0,\infty,\Omega} \int_0^T\sum_{K\in \mathcal{T}}\|c_h-c-(\bar{c}_h-\bar{c})\|_{\partial K}
      h_K^{3/2}\norm{\sigma}_{2,K} \dif t
    \\
    & \leq Ch^2\|u_h\|_{0,\infty,\Omega} \del[2]{\int_0^T \tnorm{{\bf c}-{\bf c}_h}_c^2\dif t}^{1/2}\del[2]{\int_0^T \|\sigma\|^2_{2,\Omega}\dif t}^{1/2}
    \\
    &\leq Ch^{2(k+1)} + \delta \norm{c-c_h}^2_{L^2(0,T;L^2(\Omega))}.
  \end{align*}
  Finally we consider $T_{14}$.  By \cref{eq:D_max},
  \cref{eq:inf-trace-continuous}, \cref{eq:L2projerror-2},
  \cref{cor:energynormerror}, \cref{eq:regularity}
  \begin{align*}
    T_{14}
    &\leq C\int_0^T \sum_{K\in\mathcal{T}} \frac{\beta_c}{h_K}\|D(u_h)\|_{0,\infty,\partial K}\|c_h-\bar{c}_h\|_{\partial K}h_K^{3/2}
      \|\sigma\|_{2,K} \dif t
    \\
    &\leq C\beta_c h\|D(u_h)\|_{0,\infty, \Omega}\int_0^T \sum_{K\in\mathcal{T}} h_K^{-1/2}\|c_h-\bar{c}_h\|_{\partial K}
      \|\sigma\|_{2,K} \dif t
    \\
    & \leq C\beta_c D_{{\rm max}} h \del[2]{\int_0^T \tnorm{c-c_h}_c^2\dif t}^{1/2}\del[2]{\|\sigma\|^2_{2,\Omega}\dif t}^{1/2}
    \\
    & \leq Ch^{2k} + \delta\norm{c-c_h}^2_{L^2(0,T;L^2(\Omega))}.
  \end{align*}
  Therefore, combining all bounds and picking $\delta>0$ small enough
  the result follows.  
\end{proof}

\Cref{thm:l2errorestc} shows optimal convergence for the concentration
in the $L^2$-norm for sufficiently smooth $c$, $c_0$, and $u$.

\section{Numerical Experiments}
\label{sec:numexp}
In this section we demonstrate performance of the method. We will
first verify \cref{thm:l2errorestc}, i.e., that the concentration
converges optimally in the $L^2$-norm, after which we verify
compatibility of the coupled discretization. As final test case we
consider a more realistic test case of contaminant transport and
compare our results to those obtained in literature.

All simulations have been implemented in the finite element library
NGSolve \cite{Schoberl:2014}. We use unstructured simplicial meshes to
discretize the domain $\Omega$ and use Crank--Nicolson time
stepping. Furthermore, the penalty parameters in
\cref{eq:hdgwf,eq:semidiscrete} are set to $\beta_f = 10k^2$ and
$\beta_c = 6\ell^2$.

\subsection{Rates of convergence in the $L^2$-norm}
\label{ss:rates_convergence_transport}
Let $\Omega=[0,1]^2$ with $\Omega^d = [0,1] \times [0, 0.5]$ and
$\Omega^s = [0,1] \times [0.5, 1]$. The source terms and boundary
conditions for the Stokes--Darcy problem \cref{eq:system,eq:interface}
are chosen such that the exact solution is given by
\begin{equation}
  \begin{aligned}
    u|_{\Omega_S}
    &=
    \begin{bmatrix}
      -\sin(\pi x_1)\exp(x_2/2)/(2\pi^2)
      \\
      \cos(\pi x_1)\exp(x_2/2)/\pi
    \end{bmatrix},
    &
    p|_{\Omega_S} &= \frac{\kappa\mu-2}{\kappa\pi}\cos(\pi x_1)\exp(x_2/2),
    \\
    u|_{\Omega_D}
    &=
    \begin{bmatrix}
      -2\sin(\pi x_1)\exp(x_2/2) 
      \\
      \cos(\pi x_1)\exp(x_2/2)/\pi
    \end{bmatrix},
    &
    p|_{\Omega_D} &= -\frac{2}{\kappa\pi}\cos(\pi x_1)\exp(x_2/2),
  \end{aligned}  
  \label{eq:exactsolution}  
\end{equation}
with $\alpha = \mu\kappa^{1/2}(1 + 4\pi^2)/2$ as considered also
in \cite{Cesmelioglu:2020, Correa:2009}. We take $\mu=1$ and
$\kappa=1$.

We solve the Stokes--Darcy problem using the EDG-HDG discretization
\cref{eq:hdgwf}. We then replace the exact velocity $u$ in the
transport problem \cref{eq:transport_system} by $u_h$, i.e., the
discrete velocity solution to \cref{eq:hdgwf}. Other parameters in
\cref{eq:transport_system} are set as: $\phi = 1$ and
\begin{equation}
  \label{eq:diffusiontensor}
  D =
  \begin{bmatrix}
    0.01 & 0.005 \\ 0.005 & 0.02
  \end{bmatrix}.
\end{equation}
We set the source and boundary terms such that the exact solution to
\cref{eq:transport_system} is given by
$c(x,t) = \sin(2\pi(x_1-t))\cos(2\pi(x_2-t))$.

We compute the rates of convergence of the contaminant $c$ for
$\ell = k-1 = 1$ and $\ell = k-1 = 2$ with the time step small enough
so that spatial errors dominate over temporal errors. The error in
the $L^2$-norm and rates of convergence at final time $t=1$ are
presented in \cref{tab:rates_conv_transport}. We obtain optimal rates
of convergence for $c_h$, verifying \cref{thm:l2errorestc}.
\begin{table}
  \centering {
    \begin{tabular}{ccc|ccc}
      \hline
      \multicolumn{3}{c}{$\ell=1$} & \multicolumn{3}{|c}{$\ell=2$} \\
      Elements & $\norm{c - c_h}_{\Omega}$ & Rate & Elements & $\norm{c - c_h}_{\Omega}$ & Rate  \\
      \hline
      \num{28}    & 2.2e-1 &   - & \num{8}    & 3.6e-1 &   - \\
      \num{152}   & 3.1e-2 & 2.8 & \num{28}   & 4.7e-2 & 3.0 \\
      \num{578}   & 8.5e-3 & 1.9 & \num{152}  & 3.2e-3 & 3.9 \\
      \num{2416}  & 2.0e-3 & 2.1 & \num{578}  & 3.3e-4 & 3.3 \\
      \num{9584}  & 4.5e-4 & 2.1 & \num{2416} & 3.3e-5 & 3.3 \\
      \hline
    \end{tabular}
  }
  \caption{Errors and rates of convergence in $\Omega$ for the
    solution $c_h$ of the EDG discretization of the transport equation
    \cref{eq:semidiscrete} using an approximate velocity $u_h$
    computed by the EDG-HDG method for the Stokes--Darcy system
    \cref{eq:hdgwf}. The test case is described
    in~\cref{ss:rates_convergence_transport}. }
  \label{tab:rates_conv_transport}
\end{table}

\subsection{Compatibility of the Stokes--Darcy and transport discretization}
\label{ss:testing_compatibility_hdg}
To verify compatibility of the Stokes--Darcy and transport
discretization \cref{eq:hdgwf,eq:semidiscrete}, we need to verify that
\cref{eq:semidiscrete} is able to preserve the constant solution when
$f = -\chi^df^d\tilde{c}$ where $\tilde{c}$ is a constant.

For this test case we first solve the discrete Stokes--Darcy flow
problem \cref{eq:hdgwf}. We use the same setup as in
\cref{ss:rates_convergence_transport} and compute the discrete
velocity $u_h$ and discrete pressure $p_h$ solutions on a grid
consisting of 578 elements and using $k=2$.

We then solve the discrete transport problem
\cref{eq:semidiscrete}. We set $u = u_h$, $\phi = 1$, $\ell = k-1 = 1$
and use the diffusion tensor given by \cref{eq:diffusiontensor} and
time step $\Delta t = 10^{-3}$. We choose the initial and boundary
conditions such that the exact solution is given by
$c = \tilde{c} = 1$. At final time $t=1$ we compute
$\norm{1 - c_h}_{\Omega} = 1.5\cdot 10^{-13}$, verifying compatibility
of the discretization \cref{eq:hdgwf,eq:semidiscrete}.

To compare, we now consider a discretization of the Stokes--Darcy
problem that is \emph{not} compatible with the EDG discretization
\cref{eq:semidiscrete} of the transport problem. The setup is the same
as above except that we discretize the Stokes--Darcy problem by an
embedded discontinuous Galerkin method on $\Omega^s$ for the Stokes
equations \cite{Rhebergen:2020}, and a standard L-HDG method on
$\Omega^d$ for the Darcy equations \cite{Cockburn:2009a}. We remark
that this discretization of the Stokes--Darcy problem is not exactly
mass conserving, i.e., the discrete velocity solution $u_h$ does not
satisfy the properties described in \cref{eq:uh_properties}. As a
result, this discretization of the Stokes--Darcy problem cannot be
proven to be compatible with the EDG discretization
\cref{eq:semidiscrete} of the transport problem. Indeed, computing
the error in the concentration at final time $t=1$ we find that
$\norm{1 - c_h}_{\Omega} = 2.4\cdot 10^{-4}$, i.e., this incompatible
discretization is not able to preserve the constant solution.

In \cref{fig:transport} we compare the solution of the concentration
at final time $t$ found using the compatible EDG-HDG discretization
with the solution found using the incompatible discretization. It is
clear that the incompatible discretization does not preserve the
constant $c=1$.

\begin{figure}
  \centering
  \subfloat[Compatible flow-transport discretization. \label{fig:transport_a}]
  {\includegraphics[width=0.49\textwidth]{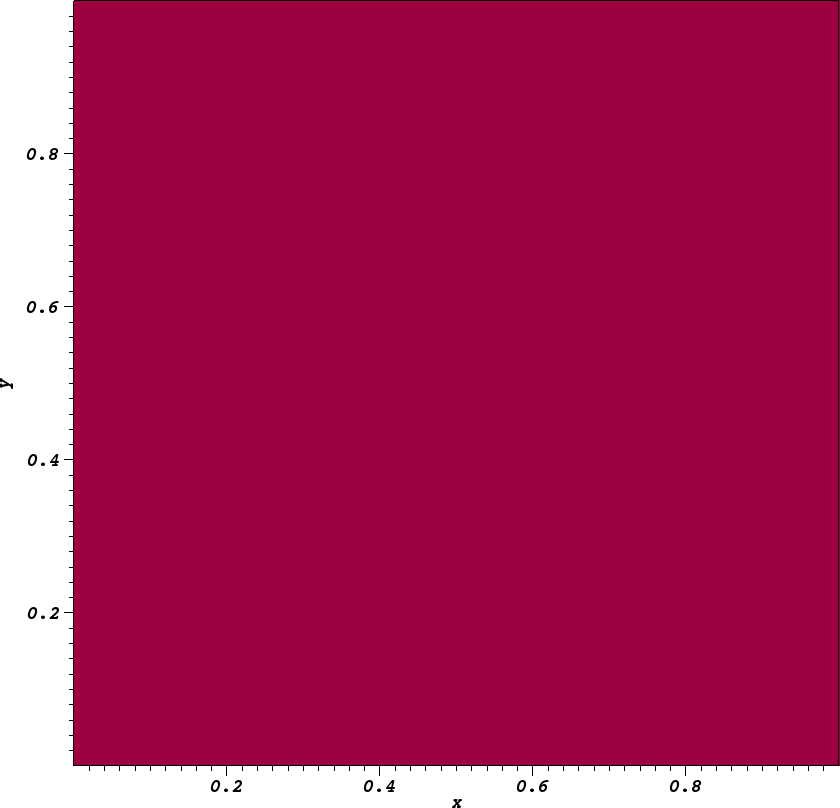}}
  \subfloat[Incompatible flow-transport discretization. \label{fig:transport_b}]
  {\includegraphics[width=0.49\textwidth]{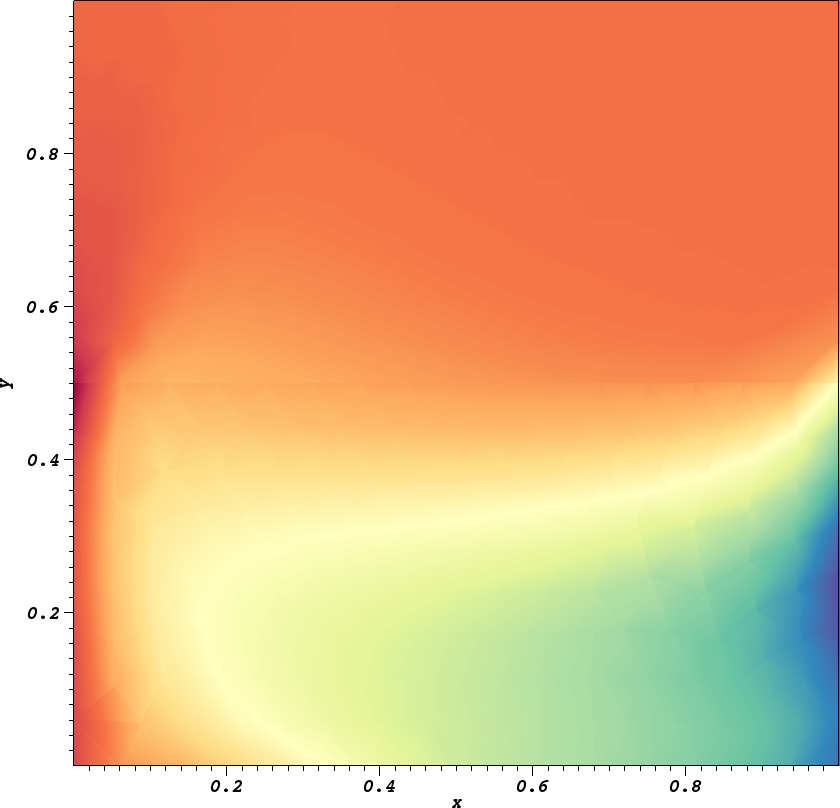}}
  \caption{Test case from~\cref{ss:testing_compatibility_hdg}.  The
    solution to the transport equation at $t=1$. Left: the solution
    $c=1$ is approximated up to machine precision using the compatible
    EDG-HDG discretization. Right: an incompatible flow-transport
    discretization cannot preserve the constant solution $c=1$. Here
    $c_h \in [0.99927, 1.00019]$.}
  \label{fig:transport}
\end{figure}

\subsection{Contaminant transport}
\label{ss:contaminant_transport}
We finally consider a simulation of coupled surface/subsurface flow
and contaminant transport. This test case is similar to that proposed
in~\cite[Example 7.2]{Vassilev:2009}.

For the Stokes--Darcy problem we consider the same setup as in
\cite[Section 6.2]{Cesmelioglu:2020}; we consider the unit square
domain $\Omega = (0,1)^2$ which is divided into a Stokes region
$\Omega^s = (0,1) \times (0.5, 1)$ that represents a lake or a river,
and a Darcy region $\Omega^d = \Omega\backslash\Omega^s$ representing
an aquifer. The domain is divided into \num{14900} simplicial
elements. The mesh is such that $\mathcal{T}^s$ is an exact
triangulation of $\Omega^s$, $\mathcal{T}^d$ is an exact triangulation
of $\Omega^d$, and element boundaries match on the interface
$\Omega^s\cap\Omega^d$. We use a time step of $\Delta t = 10^{-3}$ and
set $\ell = k-1 = 2$.

Let the boundary of the Stokes region be partitioned as
$\Gamma^s = \Gamma^s_1\cup\Gamma_2^s\cup\Gamma_3^s$ where
$\Gamma^s_1 :=\cbr[0]{ x \in\Gamma^s\ :\ x_1=0 }$,
$\Gamma^s_2 :=\cbr[0]{ x \in\Gamma^s\ :\ x_1=1 }$ and
$\Gamma^s_3 :=\cbr[0]{ x \in\Gamma^s\ :\ x_2=1 }$. Similarly, let
$\Gamma^d = \Gamma^d_1\cup\Gamma_2^d$ where
$\Gamma^d_1 :=\cbr[0]{ x \in\Gamma^d\ :\ x_1=0 \ \text{or} \ x_1=1}$ and
$\Gamma^d_2 :=\cbr[0]{ x \in\Gamma^d\ :\ x_2=0 }$. We impose the
following boundary conditions:
\begin{align*}
  u &= (x_2(3/2-x_2)/5, 0) && \text{on}\ \Gamma_1^s, \\
  \del{-2\mu\varepsilon(u) + p\mathbb{I}}n &= 0 && \text{on}\ \Gamma_2^s, \\
  u\cdot n &= 0  && \text{on}\ \Gamma_3^s, \\
  \del{-2\mu\varepsilon(u) + p\mathbb{I}}^t &= 0 && \text{on}\ \Gamma_3^s, \\
  u\cdot n &= 0 && \text{on}\ \Gamma_1^d, \\
  p &= -0.05 && \text{on}\ \Gamma_2^d.
\end{align*}
We set the permeability to
\begin{equation*}
  \kappa(x) = 700 (1 + 0.5 (\sin(10\pi x_1)\cos(20\pi x_2^2)
  + \cos^2(6.4 \pi x_1)\sin(9.2\pi x_2))) + 100.
\end{equation*}
Other parameters in \cref{eq:system,eq:interface} are set as
$\mu = 0.1$, $\alpha=0.5$, $f^s=0$, and $f^d=0$.

For the transport equation~\cref{eq:transport_system} the diffusion
tensor is set to
\begin{equation*}
  D(u_h) =
  \begin{cases}
    \delta I, & \text{ in } \Omega^s,
    \\
    \phi d_m \mathbb{I} + d_l|u_h|\mathbb{T} + d_t|u_h|(\mathbb{I} - \mathbb{T}),
    & \text{ in } \Omega^d,    
  \end{cases}
\end{equation*}
where $u_h$ is the velocity solution to \cref{eq:hdgwf}, $\delta > 0$,
$d_l, d_t\geq 0$ are longitudinal and transverse dispersivities and
$d_m> 0 $ is the molecular diffusivity, $\mathbb{I}$ denotes the
identity matrix, $\mathbb{T}= u_hu_h^T/|u_h|^2$ and $u_h^T$ is the
transpose of the vector $u_h$. This diffusion tensor satisfies
\cref{eq:D_min,eq:upperboundDuabs,eq:D_max} (assuming $d_l\geq d_t$, which is usually the
case) and \cref{eq:DLipschitz} \cite{Douglas:1983, Sun:2002}. In our
numerical example, we choose $\delta = 10^{-6}$, $\phi = 1$ on
$\Omega^s$ and $\phi = 0.4$ on $\Omega^d$ and
$d_m = d_l = d_t = 10^{-5}$. The initial condition for the plume of
contaminant is given by
\begin{equation*}
  c_0(x) =
  \begin{cases}
    0.95 & \text{if}\ \sqrt{ (x_1-0.2)^2 + (x_2-0.7)^2 } < 0.1, \\
    0.05 & \text{otherwise}.
  \end{cases}
\end{equation*}

In \cref{fig:contaminant_transport_conduc_vel} we show the
permeability and the computed velocity field (which are identical to
\cite[Figure
2]{Cesmelioglu:2020}). In \cref{fig:contaminant_transport} we show the
plume of contaminant spreading through the surface water region and
penetrating into the porous medium. As observed also in~\cite[Example
7.2]{Vassilev:2009}, the contaminant plume stays compact while in the
surface water region but spreads out in the groundwater region due to
the heterogeneity of the porous media.
\begin{figure}
  \centering
  \subfloat[The permeability. \label{fig:tc_contaminant_hc}]{\includegraphics[width=0.48\textwidth]{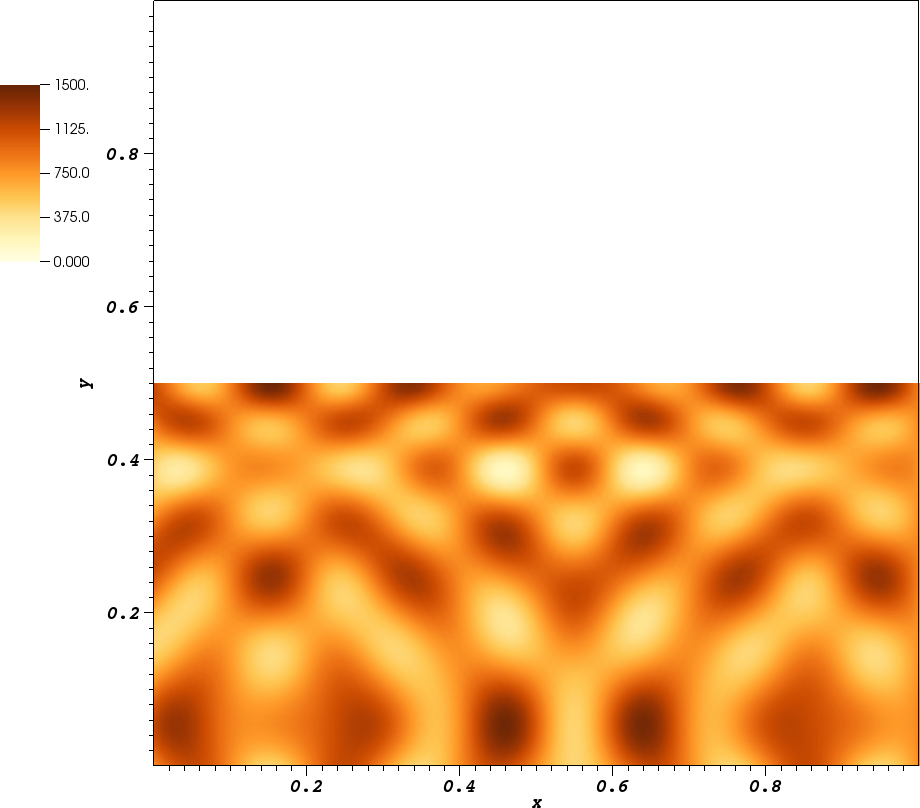}}
  \quad
  \subfloat[The velocity field. \label{fig:tc_contaminant_vel}]{\includegraphics[width=0.48\textwidth]{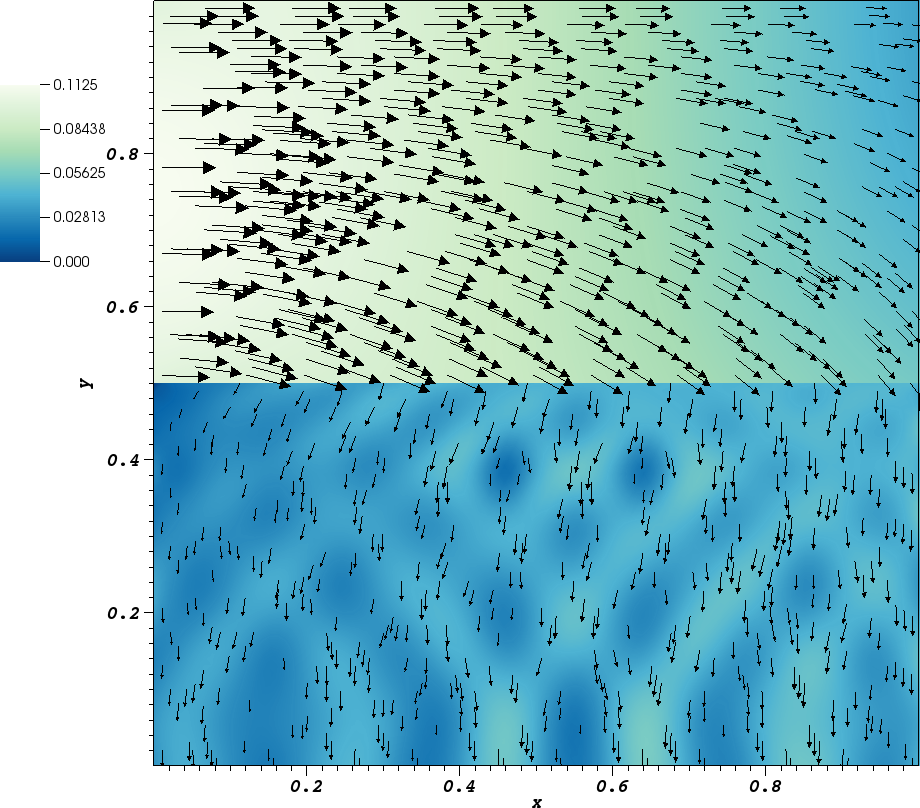}}
  \caption{The permeability and computed velocity field for the test
    case described in \cref{ss:contaminant_transport}.}
  \label{fig:contaminant_transport_conduc_vel}
\end{figure}
\begin{figure}
  \centering \subfloat[The plume at
  $t=0$. \label{fig:tc_contaminant_t0}]{\includegraphics[width=0.48\textwidth]{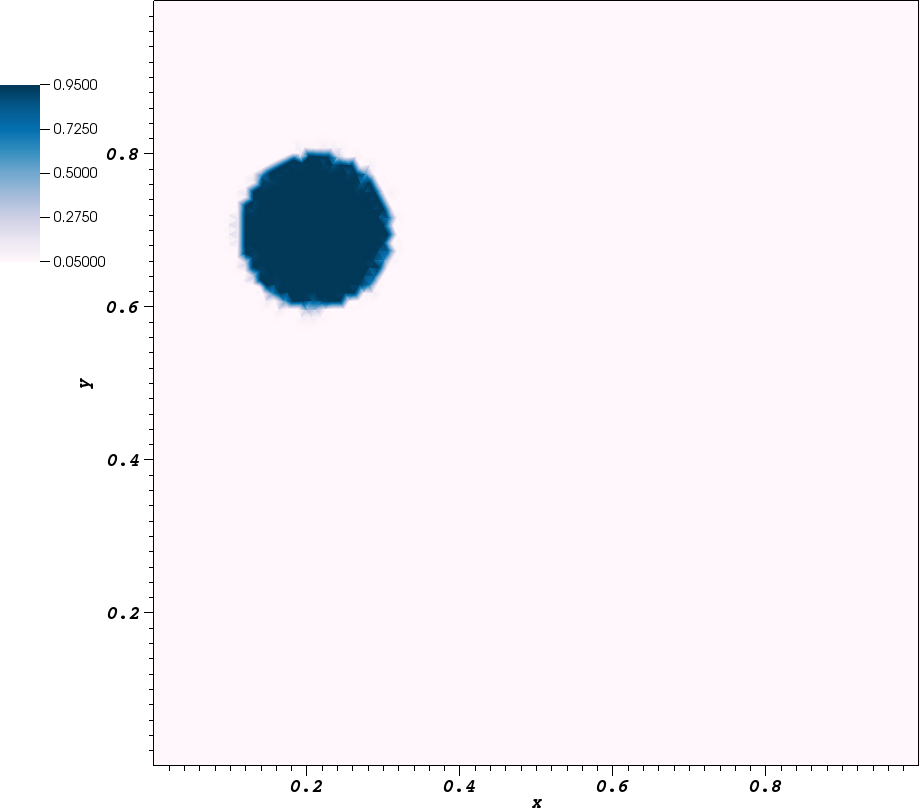}}
  \quad
  \subfloat[The plume at $t=3.3$. \label{fig:tc_contaminant_t2}]{\includegraphics[width=0.48\textwidth]{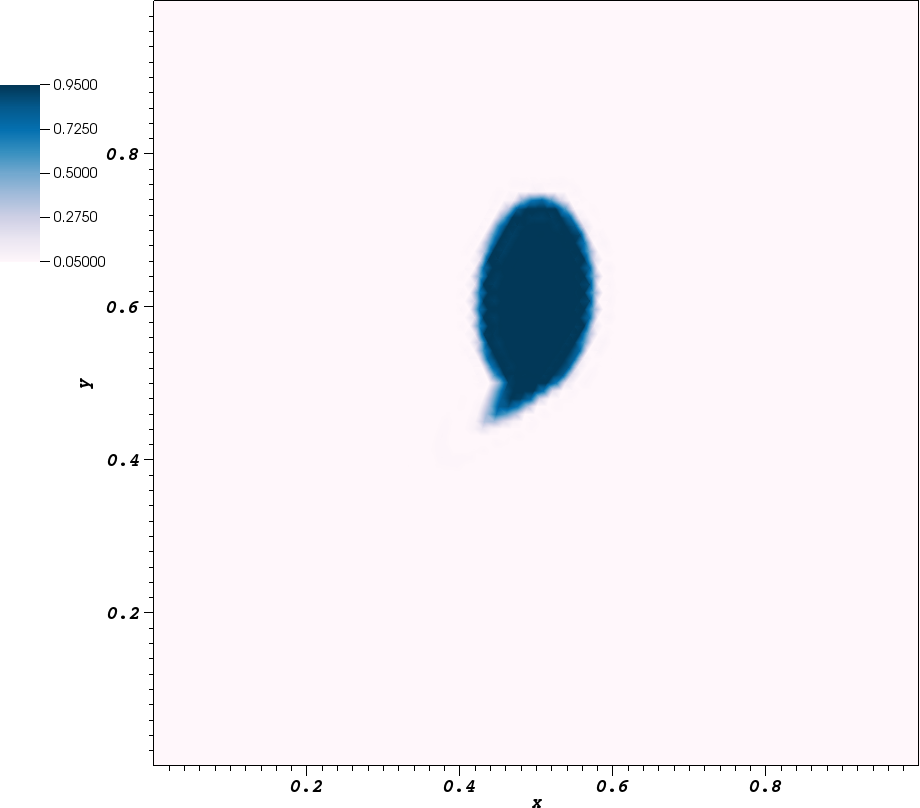}} \\
  \subfloat[The plume at
  $t=6.6$. \label{fig:tc_contaminant_t4}]{\includegraphics[width=0.48\textwidth]{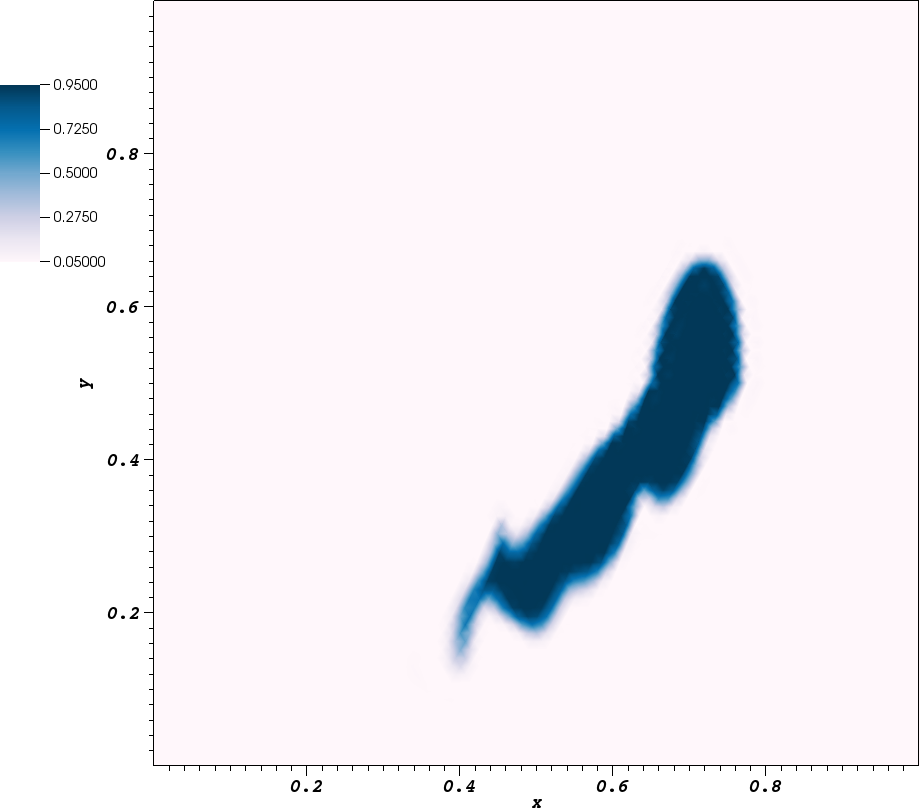}}
  \quad \subfloat[The plume at
  $t=10$. \label{fig:tc_contaminant_t6}]{\includegraphics[width=0.48\textwidth]{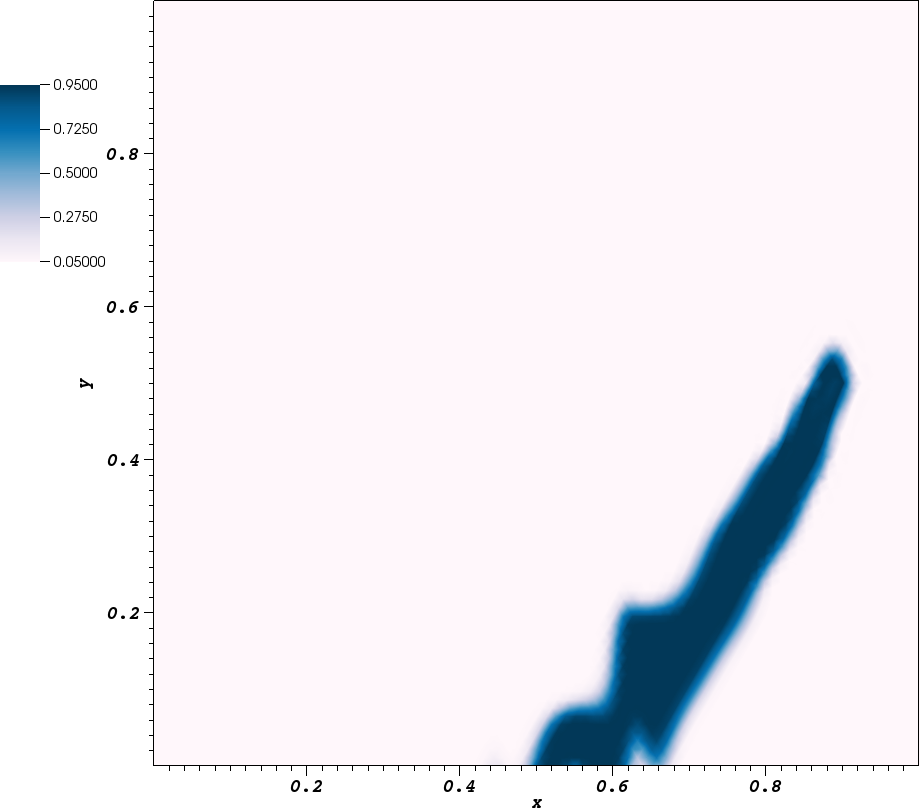}}
  \caption{The plume of contaminant spreading through the surface
    water region and penetrating into the porous medium with snapshots
    at different instances in time. The test case is described
    in~\cref{ss:contaminant_transport}.}
  \label{fig:contaminant_transport}
\end{figure}

\section{Conclusions}
\label{sec:conclusions}

We have analyzed a compatible embedded-hybridized discontinuous
Galerkin discretization for the one-way coupling between Stokes--Darcy
flow and transport and proved existence and uniqueness and optimal
convergence rates for the discrete transport problem. These results
complement our previous work in which we proved optimal and
pressure-robust error estimates for the EDG-HDG discretization of the
Stokes--Darcy system. We verified our theory by numerical examples. We
furthermore demonstrated that an incompatible discretization of the
coupled Stokes--Darcy and transport problem can result in small
oscillations in the solution to the transport equation. This shows the
importance of compatible discretizations for coupled flow and
transport problems.

\section*{Acknowledgments}
SR gratefully acknowledges support from the Natural Sciences and
Engineering Research Council of Canada through the Discovery Grant
program (RGPIN-05606-2015).

\bibliographystyle{abbrvnat}
\bibliography{references}
\end{document}